\documentclass{amsart}
\usepackage[utf8]{inputenc}
\usepackage{url}
\usepackage{hyperref}
\usepackage{amsmath,amsfonts}
\usepackage{amsthm}
\usepackage[colorinlistoftodos]{todonotes}

\newtheorem{theorem}{Theorem}[section]
\newtheorem{lemma}[theorem]{Lemma}
\newtheorem{corollary}[theorem]{Corollary}
\newtheorem{proposition}[theorem]{Proposition}

\newtheorem{conjecture}[theorem]{Conjecture}
\theoremstyle{definition}
\newtheorem{definition}[theorem]{Definition}
\newtheorem{example}[theorem]{Example}
\newtheorem{remark}[theorem]{Remark}
\newcommand{\T}{\overline{\R}}

\newcommand{\R}{\mathbb{R}}
\newcommand{\K}{\mathbb{K}}

\newcommand{\C}{\mathbb{C}}
\newcommand{\bnn}{{\binom{[n]}2}}
\newcommand{\p}{\mathbf{p}}
\newcommand{\RR}{\mathcal{R}}

\newcommand{\HH}{\mathcal{H}}
\newcommand{\MM}{\mathcal{M}}

\newcommand{\ab}{\mathbf{a}}
\newcommand{\bb}{\mathbf{b}}

\newcommand{\Sep}{\operatorname{FP}}
\newcommand{\Sepplus}{\operatorname{FP}^+}

\newcommand{\Tree}{\mathcal{T}\hskip-3pt\mathit{ree}}
\newcommand{\Gr}[2]{\mathcal{G}\hskip-1pt\mathit{r}_#1(#2)}
\newcommand{\Pf}[2]{\mathcal{P}\hskip-1pt\mathit{f}\hskip-1pt_{#1}(#2)}
\newcommand{\Ass}[2]{\mathcal{A}\hskip-0pt\mathit{ss}_{#1}(#2)}
\newcommand{\OvAss}[2]{\overline{\mathcal{A}\hskip-0pt\mathit{ss}}_{#1}(#2)}
\newcommand{\PV}[2]{\mathrm{Pf}_{#1}(#2)}
\newcommand{\PVplus}[2]{\mathrm{Pf}_{#1}^+(#2)}
\newcommand{\g}{\textbf{g}}
\newcommand{\Grob}{\operatorname{Grob}}
\newcommand{\Skew}{\operatorname{Sym}}

\DeclareMathOperator{\trop}{trop}
\DeclareMathOperator{\rank}{rank}
\DeclareMathOperator{\ini}{in}

\title{Multitriangulations and tropical Pfaffians}
\thanks{{Supported by grant PID2019-106188GB-I00 funded by MCIN/AEI/10.13039/501100011033, by FPU19/04163 of the Spanish Government, and by project CLaPPo (21.SI03.64658) of Universidad de Cantabria and Banco Santander}}

\author{Luis Crespo Ruiz \and Francisco Santos}
\address{
Departamento de Matem\'aticas, Estad\'istica y Computaci\'on,
Universidad de Canta\-bria,
39005 Santander, Spain
}
\email{francisco.santos@unican.es, luis.cresporuiz@unican.es}

\begin{document}


%

\begin{abstract}
    The $k$-associahedron $\Ass{k}{n}$ is the simplicial complex of $(k+1)$-crossing-free subgraphs of the complete graph with vertices on a circle. Its facets are called $k$-triangulations.
    We explore the connection of $\Ass{k}{n}$ with the \emph{Pfaffian variety} $\Pf{k}{n}\subset \K^\bnn$ of antisymmetric matrices of rank $\le 2k$.
 
        First, we characterize the Gr\"obner cone $\Grob_k(n)\subset\R^\bnn$ for which the
        initial ideal of $I(\Pf{k}{n})$  equals
        the Stanley-Reisner ideal of $\Ass{k}{n}$ (that is, the monomial ideal generated by $(k+1)$-crossings).
        This implies that $k$-triangulations are bases in the algebraic matroid of $\Pf{k}{n}$, a matroid closely related to low-rank completion of antisymmetric matrices.
    
        We then look at the tropicalization of $\Pf{k}{n}$ and show that $\Ass{k}{n}$ embeds naturally as the intersection of $\trop(\Pf{k}{n})$ and $\Grob_k(n)$, and that it is contained in the \emph{totally positive} part $\trop^+(\Pf{k}{n})$ of it. 
        
        We show that for $k=1$ and for each triangulation $T$ of the $n$-gon, the projection of this embedding of $\Ass{k}{n}$ to the $n-3$ coordinates corresponding to diagonals in $T$ gives a complete polytopal fan, realizing the associahedron. This fan is linearly isomorphic to the $\g$-vector fan of the cluster algebra of type $A$, shown to be polytopal by Hohlweg, Pilaud and Stella in (2018).
    
\end{abstract}
\maketitle

\tableofcontents

\section{Introduction}
\label{sec:intro}
Throughout the paper we consider $[n]=\{1,2,\dots,n\}$ as the vertex set of a complete graph $K_n$, thus calling \emph{edges} its size-two subsets. We think of the $n$ vertices as drawn in order along a circle (or any other convex closed curve in the plane) so that two edges $\{i,j\}, \{i',j'\} \in \bnn$ with $i<j$ and $i'<j'$ are said to \emph{cross} each other if either $i<i'<j<j'$ or $i'<i<j'<j$. A \emph{$k$-crossing} is a set of $k$ edges that mutually cross each other.

\subsection*{Multitriangulations}
\begin{definition}
A subset $T\subseteq \bnn$ is called $(k+1)$-free if it contains no $(k+1)$-crossing. 
The maximal $(k+1)$-free graphs are called $k$-\emph{triangulations}, or multitriangulations if $k$ is not specified.
\end{definition}

Many nice combinatorial properties of $k$-\emph{triangulations} are known \cite{PilPoc,PilSan,Stump}. For starters, they all have the same cardinality, equal to $k(2n-2k-1)$ \cite{Naka,DKM}. 
We are interested in the abstract simplicial complex $\Ass{k}{n}$ on the vertex set $\binom{[n]}2$ whose faces are $(k+1)$-free graphs.
Hence,  facets are $k$-triangulations.

If an edge $\{i,j\}$ has $|i-j|\le k$ (where indices are taken modulo $n$, and distance is measured cyclically), then it lies in every $k$-triangulation since it cannot be part of any $(k+1)$-crossing. We call these edges \emph{irrelevant} and call \emph{irrelevant face} the face of 
$\Ass{k}{n}$ they span. We can thus define the reduced complex, $\OvAss{k}{n}$, the faces of which are the $(k+1)$-free sets of \emph{relevant} edges. The exact relation between $ \Ass{k}{n}$ and $\OvAss{k}{n}$ is that the former is the join of the latter with the irrelevant face, and hence the latter is the link of the former at the irrelevant face. 
Based on the fact that  $\OvAss{1}{n}$ is the polar complex to the face poset of the standard associahedron we define:

\begin{definition}
We call $\OvAss{k}{n}$ the $k$-associahedron, or \emph{multiassociahedron} of parameters $n$, $k$.  We refer to $\Ass{k}{n}$ as the \emph{extended multiassociahedron}.
\end{definition}

What we have said so far implies that $\OvAss{k}{n}$ is a pure complex of dimension $k(n-2k-1)-1$.
Jonsson \cite{Jonsson1,Jonsson2} proved it to be a shellable simplicial sphere, and conjectured it to be polytopal. This conjecture is one of our motivations:

\begin{conjecture}[Jonsson]
\label{conj:jonsson}
For every $k\ge 1$ and $n\ge 2k+1$,
$\OvAss{k}{n}$ is isomorphic to the face lattice of a simplicial polytope of dimension 
$k(n-2k-1)$.
\end{conjecture}

Besides the case $k=1$, the conjecture is known to hold for
following cases:
    $n\le 2k+3$
    and for $n=2$ and $k=8$ there is an ad-hoc construction of the polytope~\cite{BokPil}.
    In a forthcoming paper~\cite{CreSan-associa} we show polytopality for the remaining cases with $n\le10$, namely $(k,n) \in \{(2,9),(2,10),(3,10)\}$.
For some additional cases there are constructions that realize $\OvAss{k}{n}$ as a complete simplicial fan. This happens for $n=2k+4$~\cite{BCL} and for $k=2$ and $n\le 13$ \cite{Manneville}.

 Polytopality of $\OvAss{k}{n}$ is also relevant from the perspective of Coxeter combinatorics. Let $(W,S)$ be a Coxeter system. 
  Let $w\in W$ be an element in the group and $Q$ a word of a certain length $N$ and containing a reduced expression for $w$ as a subword. The \emph{subword complex} of $Q$ and $w$ is the simplicial complex with vertex set $\{1,\dots,N\}$ consisting of subsets of positions that can be deleted from $Q$ and still contain a reduced expression for $w$. Kuntson and Miller~\cite[Theorem 3.7 and Question 6.4]{KnuMil} proved that every subword complex is either a shellable ball or sphere, and they asked whether all spherical subword complexes are polytopal. It turns out that  $\OvAss{k}{n}$ is a spherical subword complex for the Coxeter system of type $A_{n-2k-1}$~\cite[Theorem 2.1]{Stump} and, moreover, it is \emph{universal}: every other spherical subword complex of type $A$ appears as a link in some $\OvAss{k}{n}$~\cite[Proposition 5.6]{PilSan:brick}. In particular, Conjecture~\ref{conj:jonsson} would answer the question of Knutson and Miller in the positive. (Versions of multiassociahedra for the rest of finite Coxeter groups exist, with the same implications~\cite{CLS14}).

\subsection*{Pfaffians and tropical varieties}


In the case $k=1$, one way of realizing the associahedron is as the positive part of the space of ``tree metrics'', which coincides with the tropicalization $\trop(\Gr2n)$ of the Grassmannian $\Gr2n$ (see \cite{SpeStu,SpeWil}, or Remark~\ref{rem:tree_metrics}). More precisely:

\begin{theorem}[\protect{\cite[Section 5]{SpeWil}}]
The totally positive tropical Grassmannian $\trop^+(\Gr2n)$ is a simplicial fan isomorphic to (a cone over) the extended associahedron $\OvAss{1}{n}$.
\end{theorem}

Let us briefly recall what the tropicalization of a variety, and its positive part, are. (See also \cite{BLS}). Let $I\subset\K[x_1,\dots,x_N]$ be a polynomial ideal and let $V=V(I)\subset \K^N$ be its corresponding variety. Each vector $v\in \R^N$, considered as giving weights to the variables, defines an initial ideal $\ini_v(I)$, consisting of the initial forms $\ini_v(f)$ of the polynomials in $f$. For the purposes of this paper we take the following definitions. (These are not the standard definitions, but are equivalent to them as shown for example in \cite[Propositions 2.1 and 2.2]{SpeWil}):

\begin{definition}
\label{defi:tropical}
The \emph{tropical variety} $\trop(V)$ of $V$ equals the set of $v\in \R^N$ for which $\ini_v(I)$ does not contain any monomial. If $\K=\C$, the \emph{totally positive part} of $\trop(V)$, denoted $\trop^+(V)$,  equals the set of $v\in \R^N$ for which $\ini_v(I)$ does not contain any polynomial with all coefficients real and positive.
\end{definition}

Pachter and Sturmfels~\cite[p.~107]{PacStu} hint at the fact that the relation between the associahedron and $\Gr2n$ extends to a relation between the multiassociahedron $\Ass{k}{n}$ and the tropical variety of Pfaffians of  degree $k+1$. 
Recall that a \emph{Pfaffian of degree $k$} is the square root of  the determinant of an antisymmetric matrix $M$ of size $2k$.
Considering the entries of  $M$ as indeterminates (over a certain field $\K$), the Pfaffian is a homogeneous polynomial of degree $k$ in $\K[x_{i,j}, \{i,j\} \in \binom{[2k]}2]$, with one monomial for each of the $(2k-1)!!$  perfect matchings in $[2k]$ (see Section \ref{sec:pfaffians}). 

For each $n\ge 2k+2$, let $I_k(n)$ be the ideal in $\K[x_{i,j}, \{i,j\} \in \binom{[n]}2]$
generated by all the Pfaffians of degree $k+1$.
Let $\Pf{k}{n}\subset\K^{\binom{n}{2}}$ be the corresponding algebraic variety. 
That is, points in $\Pf{k}{n}$ are antisymmetric $n\times n$ matrices  with coefficients in $\K$ and of rank at most $2k$. 
It is well-known and easy to see that $\Pf{1}{n}$ equals the Grassmannian $\Gr2n$ in its Pl\"ucker embedding and, as pointed out in \cite{PacStu}, 
$\Pf{k}{n}$ equals the $k$-th secant variety of it.

For $k=1$,  Pfaffians are a universal Gr\"obner basis of $I_k(n)$ \cite{PacStu,SpeStu}. For $k>1$ they are not (see Example~\ref{exm:UGB}), but it is known that they are a Gr\"obner basis for certain choices of monomial orders: in \cite{HerTru} it is proved that this happens for a $v$ that selects as initial monomial in each Pfaffian the $(k+1)$-nesting and in \cite{JonWel} for one that selects the $(k+1)$-crossing.

\subsection*{This paper}

We explore relations between $k$-triangulations and the variety $\Pf{k}{n}$. Our starting point is 
restricting Gr\"obner bases and tropicalization to weight vectors satisfying the following ``four-point positivity'' conditions.

\begin{definition}
\label{defi:fpp}
We say that a weight vector $v\in \R^{\bnn}$ is \emph{four-point positive}
(abbreviated \emph{fp-positive}) 
if for all $1\le a<a'<b<b'\le n$ we have that
\begin{align}
v_{a,b} + v_{a',b'} \ge \max\{v_{a,a'} + v_{b,b'}, v_{a,b'} + v_{a',b}\}.
\label{eq:fpp}
\end{align}
\end{definition}

We denote by $\Sep_n$ the subset of $\R^{\bnn}$ consisting of fp-positive vectors.
That is to say, $v\in \Sep_n$ if the maximum weight given by $v$ to the three matchings among four points is attained always for the matching that forms a $2$-crossing.

Although the polyhedron $\Sep_n\subset\R^\bnn$ of fp-positive vectors (the solution set of equations~\eqref{eq:fpp}) is defined by $2\binom{n}4$ inequalities, the following  $\binom{n}{2} - n$ alone are an irredundant description of it:
\begin{align}
v_{a,b}+v_{a+1,b+1}-v_{a,b+1}-v_{a+1,b}\ge 0, \quad \forall \{a,b\} \in \bnn \text{ with $|a-b|>1$},
\label{eq:sep}
\end{align}
with indices considered cyclically. 
The left-hand side vectors (that is, the  facet normals of $\Sep_n$) are linearly independent, so that $\Sep_n$ is linearly isomorphic to an orthant plus a lineality space of dimension $n$.  
We like to think of $\Sep_n$ as the ``positive orthant'' of $\R^\bnn$ regarding Pfaffians. It can be interpreted as the space of weights that represent \emph{separation vectors} among sides of the $n$-gon, or as the weights that are monotone with respect to crossings among perfect matchings of each fixed even set $U\subset [n]$. See Proposition \ref{prop:fpp} and Corollary~\ref{coro:fpp} for details.

Algebraically, fp-positive vectors are the monomial weight vectors for which the leading form of every 3-term Pl\"ucker relation 
\[
x_{a,b}  x_{a',b'} - x_{a,a'} x_{b,b'} - x_{a,b'} x_{a',b},
\qquad
1\le a<a'<b<b'\le n,
\]
contains the crossing monomial.
These relations generate the ideal of the Grassmannian $\Gr{2}{n}$. In particular, fp-positive vectors are the (closed)  Gr\"obner cone of $\Gr{2}{n}$ producing as initial ideal the one generated by $2$-crossings $x_{a,b}  x_{a',b'}$.

Extending this, we denote by $\Grob_k(n)\subset \R^{\bnn}$ the Gr\"obner cone consisting of weights that select the $(k+1)$-crossing as the leading monomial (or as one of them) in every Pfaffian of degree $k+1$. What we say above can then be stated as $\Grob_1(n) = \Sep_n$, and the result of \cite{JonWel} says that $\Grob_k(n)$ has non-empty interior. Our main results in Section~\ref{sec:Sn} are that $\Sep_n\subset \Grob_k(n)$ (Theorem \ref{thm:groebner}) plus  an explicit description of $\Grob_k(n)$, both by inequalities and by generators (Theorem \ref{thm:groebner-cone}).

\begin{theorem}[Theorem \ref{thm:groebner-cone}]
\label{thm:groebner-intro}
  For any $k>2n+2$, $\Grob_k(n)\subset \R^\bnn$ is a simplicial cone with a lineality space of dimension $n$.
\begin{enumerate}
\item  
It is generated by:
  \begin{itemize}
      \item (lineality space) For each $i\in [n]$, the line generated by the indicator vector of the set $\{ \{i,j\} : j\in [n]\setminus i\}$.
      \item (``short'' generators) For each $\{i,j\}\in [n]$ with $1\le |i-j|\le k$, the negative basis vector corresponding to $\{i,j\}$.
      \item (``long'' generators) For each $\{i,j\}\in [n]$ with $ |i-j|\ge k+2$, the ray of $\Sep_n$ corresponding to $\{i,j\}$.
  \end{itemize}
    \item  An irredundant facet description of it is given by the following $\bnn-n$ inequalities:
      \begin{itemize}
      \item (``long'' inequalities)  For each $\{i,j\}\in [n]$ with $ |i-j|\ge k+1$, the inequality \eqref{eq:sep} corresponding to $\{i,j\}$.
      \item (``short'' inequalities) For each $\{i,j\}\in [n]$ with $2\le |i-j|\le k$, the sum of the inequalities \eqref{eq:sep} corresponding to all the $\{i',j'\}$ with $|i'-j'|\le k+1$ and with $\{i,j\}$ contained in the short side of $\{i',j'\}$.
      \end{itemize}
    In particular, $\Grob_k(n)$   contains $\Sep_n$ for every $k$ and $n$.
    \end{enumerate}
\end{theorem}

This description has the following combinatorial interpretation:
modulo its lineality space (of dimension $n$, equal to that of $\Sep_n$),  $\Grob_k(n)$ is a simplicial cone with one facet and generator corresponding to each of the $\bnn -n$ edges of length at least two. The ``long'' facet-inequalities (those corresponding to relevant edges) are the same as the corresponding ones in $\Sep_n$, and the ``short'' ones are looser in $\Grob_k(n)$ than in $\Sep_n$. 

Moreover, we show that the monomial initial ideal of $I_k(n)$ produced by any generic weight vector $v\in\Grob_k(n)$ equals the Stanley-Reisner ideal of $\Ass{k}{n}$. That is, to say, the ideal in $\K[x_{i,j}, \{i,j\} \in \binom{[2k]}2]$ generated by  $(k+1)$-crossings. This, in turn, implies that $k$-triangulations are bases of the algebraic matroid of $\Pf{k}{n}$ (Corollary~\ref{coro:bases}). We find this of interest for two reasons (see Section~\ref{sec:matroid} for details):

On the one hand, the algebraic matroid $\MM(\Pf{k}{n})$ of $\Pf{k}{n}$ is closely related to low-rank completion of antisymmetric matrices~\cite{Bernstein,KRT}: given a subset $T\subset\bnn$ of positions for entries in an antisymmetric matrix $M$ of size $n\times n$, a generic choice of values for those entries can be extended to an antisymmetric matrix of rank $\le 2k$ if and only if $T$ is independent in $\MM(\Pf{k}{n})$. Thus:

\begin{theorem}
\label{thm:matroid}
	Let $T\subset \bnn$.
	\begin{enumerate}
		\item If $T$  is $(k+1)$-free and $\K$ is algebraically closed, then for any generic choice of values $v\in \K^T$ there is at least one skew-symmetric matrix of rank $\le 2k$ with the entries prescribed by $v$.
		\item If $T$ contains a $k$-triangulation then for any choice of values $v\in \K^T$ there is only a finite number (maybe zero) of skew-symmetric matrices of rank $\le 2k$ with those prescribed entries.
	\end{enumerate}
\end{theorem}

On the other hand, the algebraic matroid of $\Pf{k}{n}$ coincides with the generic \emph{hyperconnectivity matroid} in dimension $2k$ defined by Kalai \cite{Kalai}. The fact that $k$-triangulations are bases in it is closely related to the conjecture by Pilaud and Santos~\cite{PilSan} that they are bases in the generic bar-and-joint rigidity matroid in dimension $2k$ (Conjecture~\ref{cor:bar-and-joint}).

\medskip

In Section \ref{sec:Vn} we turn our attention to the tropicalization  of $\Pf{k}{n}$.
More precisely, we denote $\PV{k}{n}\subset \R^\bnn$ the intersection of the tropical hypersurfaces corresponding to Pfaffians of degree $k$. This is by definition a tropical \emph{prevariety}. It contains the tropical variety $\trop(\Pf{k}{n})$, but it does not, in general, coincide with it, as we show in Theorem~\ref{thm:ranks}.

In the light of Theorem \ref{thm:groebner-intro}, it makes sense to look at the part of $\PV{k}{n}$ contained in the Gr\"obner cone $\Grob_k(n)$. That is, we define
\[
\PVplus{k}{n}:= \PV{k}{n} \cap \Grob_k(n).
\]
 Since the crossing monomial is the only positive monomial in each 3-term Pl\"ucker relation, for $k=1$ we have
\[
\trop^+(\Pf{1}{n}) = \trop(\Pf{1}{n}) \cap \Sep_n = \PVplus{1}{n}.
\]

We partially generalize this to higher $k$.
Theorem \ref{thm:main} says that for a $v\in \Grob_k(n)$, being in $\PV{k}{n}$ is equivalent to the fact that the ``long inequalities'' of Theorem \ref{thm:groebner-intro} (that is, the inequalities \eqref{eq:sep} for $|a-b|\ge k+1$)  are satisfied with equality except in a $(k+1)$-free set. Moreover, when this happens $v$ can be proved to be in $\trop(\Pf{k}{n})$, and in fact in $\trop^+(\Pf{k}{n})$ (Corollary~\ref{cor:balanced}).
This implies our main  result:

\begin{theorem}[See Theorem~\ref{thm:main} and Corollary~\ref{cor:balanced}]
\label{thm:V+}
 \ \null
 
 \begin{enumerate}    
    \item $\PVplus{k}{n} = \Grob_k(n)\cap \trop(\Pf{k}{n}) \subset \trop^+(\Pf{k}{n})$.
    
    \item $\PVplus{k}{n}$ is the union of the faces of $\Grob_k(n)$ corresponding to $(k+1)$-free graphs. 
\end{enumerate}
\end{theorem}

In part (2), by \emph{the face corresponding to} a certain graph $G\subset \bnn$ we mean the intersection of the facets of $\Grob_k(n)$ corresponding to $\bnn \setminus G$ in the description of Theorem~\ref{thm:groebner-intro}. That is, we consider $\Grob_k(n)$ as (a cone over) the simplex with vertex set $\bnn$, so that every simplicial complex on $\bnn$ is a subcomplex of its face complex. Hence, Theorem~\ref{thm:V+} has the following interpretation:

\begin{corollary}
\label{cor:V+}
As a simplicial fan and modulo its lineality space, $\PVplus{k}{n}= \Grob_k(n)\cap \trop(\Pf{k}{n})$ is isomorphic to (the cone over) the 
extended multiassociahedron $\Ass{k}{n}$.
\end{corollary}

\begin{remark}
$\PVplus{k}{n}$ is \emph{not} equal to $\trop^+(\Pf{k}{n})$.
Put differently, ``four point positivity'' implies but is \emph{not the same} as positivity in the sense of Definition~\ref{defi:tropical}. See Example~\ref{example:positive}.
\end{remark}

Theorem~\ref{thm:V+} suggests that one way to realize the multiassociahedron as a polytope would be to find a projection $\R^{\binom{[n]}2} \to \R^{k(2n-2k-1)}$ that is injective on $\PVplus{k}{n}$. This would embed $\Ass{k}{n}$ as a full-dimensional simplicial fan in $\R^{k(2n-2k-1)}$ whose link at the irrelevant face would necessarily realize the  multiassociahedron $\OvAss{k}{n}$ as a complete fan in $\R^{k(n-2k-1)}$. A second step is needed in order to show polytopality: to find appropriate right-hand sides showing that the complete fan is polytopal.

We have achieved both steps for $k=1$. We show that, for any seed triangulation $T$, the projection $\R^{\binom{[n]}2} \to \R^{2n-3}$ that keeps only the coordinates  corresponding to edges in $T$ is injective on $\PVplus{1}{n}$ (Corollary \ref{coro:fan})
and we have found explicit right-hand sides that show the image to be polytopal (Proposition \ref{prop:poly}).
This gives us the following nice embeddings of the associahedron. 

\begin{theorem}[Propostion~\ref{prop:poly}]
\label{thm:catalan-many}
For each seed triangulation $T$ of the $n$-gon, projection of $\PVplus{1}{n}$ to the $n-3$ coordinates of the diagonals in $T$ gives a realization of the $(n-3)$-associahedron in $\R^{n-3}$ with the following properties:
\begin{enumerate}
    \item It has $n-3$ pairs of parallel facets, each pair consisting of the facet corresponding to a diagonal $\{a,b\}$ of $T$ and its rotation $\{a+1,b+1\}$.
    \item Taking the normals to those $n-3$ pairs as a linear basis for $\R^{n-3}$, all other facet normals lie in $\{0,\pm1\}^{n-3}$.
\end{enumerate}
\end{theorem}

An analysis of the rays generating  these fans reveals that they coincide with the $\g$-vector fans of the cluster algebra of type $A$, which were shown to be polytopal by Hohlweg, Pilaud and Stella in \cite{HPS}. That is, our Theorem~\ref{thm:catalan-many} recovers polytopal realizations of the $\g$-vector fans in type $A$.

\subsubsection*{Acknowledgement:} We want to thank Michael Joswig, Jean-Philippe Labb\'e, Vincent Pilaud and Benjamin Schr\"oter for useful comments on an earlier version of this paper and for pointing us to missing relevant literature.

\section{The variety of antisymmetric matrices of bounded rank}
\label{sec:Sn}

\subsection{Matchings and the Pfaffian of an antisymmetric matrix}
\label{sec:pfaffians}


The complete graph on a set of vertices $U\subset [n]$ of size $2k$ it has $(2k-1)!!$ matchings (by which we always mean a \emph{perfect} matching),  one of which is the unique $k$-crossing with vertex set $U$.
%
The \emph{parity} of a matching $E$ is the parity of the number of pairwise crossings among the edges in $E$. This parity coincides with the parity as a permutation, when the pairs of  matched vertices are written one after another, in increasing order within each pair.
%
By \emph{swapping} two pairs $\{a,b\}$ and $\{c,d\}$ in a matching $E$ we mean removing them and inserting one of the other two matchings of $\{a,b,c,d\}$ instead. Observe that one of the three matchings of $\{a,b,c,d\}$ has a crossing (that is, it is odd) and the other two are crossing-free (hence even).

\begin{lemma}
\label{lema:swap}
A swap changes parity if and only if one of the two pairs of edges in the swap (the pair removed or the pair inserted) is a crossing.
\end{lemma}

\begin{proof}
Let $\{a,b\}$ and $\{c,d\}$ be the initial pairs and $\{a,d\}$ and $\{b,c\}$ the new pairs. Any 
other edge from the matching crosses the cycle $abcda$ an even number of times. Hence, the only change in the number of crossings comes from whether the edges in the swap cross.
\end{proof}

%
%

Recall that an antisymmetric matrix of odd size $n$ has zero determinant because
\[
\det(M) = \det(M^t) = \det(-M) = (-1)^n \det(M).
\]
For even size there is the following classical result:

\begin{theorem}[Cayley 1852~\cite{Cayley}] \label{pfaffian}
	Let $M$ be a size $2k$ antisymmetric matrix. Then
	\begin{equation}\label{det}
	\det M=\left(\sum_{E\text{ matching}}s(E)\prod_{(i,j)\in E,i<j}m_{ij}\right)^2
	\end{equation}
	where the sum is taken over the matchings of $[2k]$ and $s(E) = \pm1$ according to the parity of $E$.
\end{theorem}

The expression inside the parenthesis in this theorem, that is, the square root of the determinant of an antisymmetric matrix, is called the \emph{Pfaffian} of $M$.

%

\subsection{Four-point positive weight vectors}

Let $I_k(n)\subset\K[x_{i,j}, \{i,j\}\in \bnn]$ be the ideal generated by all Pfaffians of degree $k$ of an antisymmetric matrix of size $n\times n$ (with indeterminate coefficients) and let $\Pf{k}{n}$ be the corresponding algebraic variety, whose points are the antisymmetric matrices of rank at most $2k$.

We now introduce certain term orders for the variables that produce as initial ideal of $I_k(n)$ the monomial ideal generated by $(k+1)$-crossings. For this, we need to introduce a change of basis in $\R^{\binom{[n]}2}$, and a change of point of view on the $n$-gon.

Let us call $a$-th side of the $n$-gon the edge $\{a-1,a\}$ (with indices taken modulo $n$). Then, any choice of real numbers $w_{i,j}$ (with $\{i,j\}\in \binom{[n]}2$) for the edges connecting \emph{vertices} of the $n$-gon induces a ``separation'' distance between each pair of \emph{sides}, as the sum of $w$'s of the edges separating those sides. 
That is:

\begin{definition}
\label{defi:d(w)}
	Given a vector $w\in \R^{\binom{[n]}{2}}$,  
	the \emph{separation vector} $d(w)\in \R^{\binom{[n]}{2}}$ induced by $w$ is defined as \begin{equation}
	\label{eq:w_to_d}
	d_{a,b}(w)=\sum_{\substack{(i,j)\in \binom{[n]}2\\
	a\le i< b\le j< a}}w_{ij},
	\qquad  \forall \{a,b\}\in \binom{[n]}2,
	\end{equation}
    Here the order symbols ``$<$'' and ``$\le$'' for indices are considered cyclically. E.g., $a<b<c<a$ means that $a,b,c$ are different and they appear in that cyclic order along the $n$-gon.
\end{definition}

Figure~\ref{fig:separation} shows an example of this transformation. To compute $d_{26}(w)$, where $a=2$ and $b=6$ denote two sides of the octagon, we have to sum the $w_{ij}$s in the complete bipartite graph on the two subsets of vertices separated by $a$ and $b$. 

\begin{figure}[htb]
	\begin{tikzpicture}[inner sep=0.2ex]
	\node[label=0:0] (p0) at (0:1.8cm) {$\bullet$};
	\node[label=45:1] (p1) at (45:1.8cm) {$\bullet$};
	\node[label=90:2] (p2) at (90:1.8cm) {$\bullet$};
	\node[label=135:3] (p3) at (135:1.8cm) {$\bullet$};
	\node[label=180:4] (p4) at (180:1.8cm) {$\bullet$};
	\node[label=225:5] (p5) at (225:1.8cm) {$\bullet$};
	\node[label=270:6] (p6) at (270:1.8cm) {$\bullet$};
	\node[label=315:7] (p7) at (315:1.8cm) {$\bullet$};
	\draw (p0.center)--(p4.center) (p0.center)--(p3.center)--(p1.center) (p0.center)--(p5.center)--(p7.center);
		\draw (p0.center)--(p2.center)--(p1.center) (p4.center)--(p1.center)--(p5.center) (p4.center)--(p7.center)--(p3.center)
		(p7.center)--(p2.center)--(p6.center)--(p5.center)
		(p3.center)--(p6.center)--(p4.center);

	\draw[color=red] (70:1.8cm) node[label=70:{$a=2$}] {$\bullet$} -- (250:1.8cm) node[label=250:{$b=6$}] {$\bullet$};
	\end{tikzpicture}
\caption{}
\label{fig:separation}
\end{figure}

The entries of $d(w)$ are going to be used as weights for variables in our monomial orders, but we want to have in mind the weight vector $w$ from which they come. This is well-defined thanks to the following result, which implies that the transformation from $w$ to $d(w)$ is a linear isomorphism in $\R^\bnn$:

\begin{proposition}
\label{prop:d_to_w}
    For any $w\in \R^{\binom{[n]}2}$, and every $\{a,b\}\in \binom{[n]}2$, we have
    \begin{equation}
	\label{eq:d_to_w}
    2w_{a,b}=
    d_{a,b}(w)+d_{a+1,b+1}(w)-d_{a,b+1}(w)-d_{a+1,b}(w),
	\end{equation}
    where $d_{a,a}(w)=0$ by convention.  
    
    Hence, each $v\in \R^{\binom{[n]}2}$ can be expressed uniquely as $d(w)$ for a certain $w\in \R^{\binom{[n]}2}$.
\end{proposition}
\begin{proof}
	It is enough to check that the rest of $w_{ij}$ cancel out when $d_{a,b}(w)+d_{a+1,b+1}(w)-d_{a,b+1}(w)-d_{a+1,b}(w)$ is computed via Eq.\eqref{eq:w_to_d}.
\end{proof}

That is, we can think of $d(w)$ and $w$ as different choices of linear coordinates for $\R^\bnn$.

\begin{proposition}
\label{prop:fpp}
Let $v\in \R^\bnn$ be a weight vector. The following conditions are equivalent:
\begin{enumerate}
    \item $v\in \Sep_n$. That is, it satisfies the positive four-point conditions \eqref{eq:fpp} in Definition~\ref{defi:fpp}.
    \item $v$ satisfies the $\binom{n}{2}-n$ inequalities  \eqref{eq:sep}.
    \item $v=d(w)$ in the sense of Definition \ref{defi:d(w)} for a $w$ with $w_{a,b}\ge 0$ for all $\{a,b\}\in \bnn$ with $|a-b|>1$.
    \item For every $k\ge 1$ and every $U\in \binom{[n]}{2k}$ the weights given by $v$ to matchings in $U$ are monotone with respect to swaps that create crossings.
    \item For every $k\ge 1$ and every $U\in \binom{[n]}{2k}$ the maximum weight given by $v$ to matchings in $U$ is attained by the $k$-crossing.
\end{enumerate}
\end{proposition}

\begin{proof}
    The equivalence of parts 1 and 4 is obvious and the equivalence of 2 and 3 follows from Proposition~\ref{prop:d_to_w}. The implications $5 \Rightarrow 1 \Rightarrow 2$ are also trivial because the inequalities in condition 1 are nothing but the case $k=2$ of condition 5, and they contain the inequalities in condition 2 as a subset.
    
    The implication $4  \Rightarrow 5$ follows from the fact that every matching can monotonically be converted into a full crossing by swaps that create crossings. 
    
    Finally, the implication $3 \Rightarrow 4$ follows from the fact that if
    $1\le a < a' < b < b' \le n$ then Equations~\eqref{eq:w_to_d} give
    \begin{align*}
    v_{a,b} + v_{a',b'} &= W_1+W_2+W_3,
    \\
    v_{a,a'} + v_{b,b'} &=  W_1+W_2,
    \\
    v_{a,b'} + v_{a',b} &=  W_1+W_3,
    \end{align*}
    where
    \begin{align*}
    W_1 &= 
    \sum_{\substack{	a\le i< a'\le j< b}}w_{ij} +
    \sum_{\substack{	a'\le i< b\le j< b'}}w_{ij} +
    \sum_{\substack{	b\le i< b'\le j< a}}w_{ij} +
    \sum_{\substack{	b'\le i< a\le j< a'}}w_{ij},
    \\
    W_2 &= 
    \sum_{\substack{	a\le i< a',\\ b\le j< b'}}w_{ij},
    \\
    W_3 &= 
    \sum_{\substack{	a'\le i< b,\\  b'\le j< a}}w_{ij}.
    \end{align*}
   Since $w$ is nonnegative (except perhaps for consecutive indices) we have that $W_2, W_3\ge 0$ and hence  $v_{a,b} + v_{a',b'}$ is greater or equal than both of $v_{a,a'} + v_{b,b'}$ and $
    v_{a,b'} + v_{a',b}$.
\end{proof}

That is to say, $\Sep_n$ is essentially the positive orthant in the $w$ coordinates, except for one detail.
Proposition~\ref{prop:d_to_w} implies that the  inequalities  \eqref{eq:sep} from the introduction are equivalent to
\[
w_{a,b}\ge 0 \qquad \forall \{a,b\}\in \bnn \text{ with $|a-b| > 1$};
\]
but the inequalities $w_{a,a+1}\ge 0$ are not valid in $\Sep_n$. The $n$-dimensional subspace generated by the vectors with $w_{a,b}= 0$ if $|a-b|>1$ and $w_{a,a+1}$ arbitrary can be thought of as the ``irrelevant'' part of the $w$ coordinates; in fact, is the lineality space of $\Sep_n$. 
This suggests we give it a name. We denote:

\begin{align*}
L_n&:= \left\{d(w): w\in \R^{\bnn} \text{ and $w_{i,j}=0$ if $|i-j|>1$}\right\} \cong \R^n, \\
\Sepplus_n&:=  \left\{d(w): w\in \R_{\ge 0}^{\bnn} \right\} \cong \R_{\ge 0}^{\bnn}.
\end{align*}

\begin{corollary}
\label{coro:fpp}
    $\Sep_n = L_n + \Sepplus_n$, and it is linearly isomorphic to $\R^n \times \R_{\ge 0}^{\bnn-n}$.
\end{corollary}

\begin{proof}
    By Proposition \ref{prop:d_to_w} the map $w\to d(w)$ is a linear automorphism in $\R^\bnn$; by Proposition~\ref{prop:fpp}, $\Sep_n$ is the image $\Sepplus_n$ of the positive orthant plus the linear subspace $L_n$.
\end{proof}

\subsection{Pfaffians as a Gr\"obner basis for four-point positive weights}

\begin{lemma}
Let $v\in \Sep_n$ be an fp-positive vector and assume it sufficiently generic. Consider the term order in $\K[x_{a,b}: \{a,b\} \in \binom{[n]}2]$ obtained giving each variable weight $v_{a,b}$. Then:

\begin{enumerate}
    \item The leading term in each Pfaffian of degree $k+1$ is the  monomial corresponding to the $(k+1)$-crossing among the corresponding $2k+2$ points.
    \item The leading term of the $S$-polynomial of every two Pfaffians contains a $(k+1)$-crossing.
\end{enumerate}
\end{lemma}

\begin{proof}
Part (1) is nothing but Proposition~\ref{prop:fpp}(5), taking into account that if $v$ is sufficiently generic then it cannot give the same weight to two different matchings.

    For part (2), first observe that, by the same argument as above, the second term in each Pfaffian is a matching obtained by a swap of two consecutive edges in the corresponding $(k+1)$-crossing.
    
    Let $F_0$ and $G_0$ be two $k+1$-crossings, which are the leading terms of two Pfaffians $F$ and $G$. 
    
    Suppose that $|F_0\cap G_0|=l$, and let
    \[F_0=\{\{a_1,b_1\},\ldots,\{a_l,b_l\},\{c_1,d_1\},\ldots,\{c_{k+1-l},d_{k+1-l}\}\}\]
    \[G_0=\{\{a_1,b_1\},\ldots,\{a_l,b_l\},\{e_1,f_1\},\ldots,\{e_{k+1-l},f_{k+1-l}\}\}\]
    with $a_1<a_2<\ldots<a_l$, and the same for the rest. The leading term in the $S$-polynomial is obtained as a swap of one of the $(k+1)$-crossings, together with the $k+1-l$ remaining edges of the other matching. 
    
    Without loss of generality, suppose that it is a swap of $F_0$. It cannot be a swap of the first $l$ edges, because it would cancel with a swap of $G_0$. If the swap is between $\{c_i,d_i\}$ and $\{c_{i+1},d_{i+1}\}$, the new set contains the $(k+1)$-crossing $G_0$. 
    So, let's suppose that the swap is between $\{a_l,b_l\}$ and $\{c_1,d_1\}$.
    
    \textbf{Claim:} \textit{
    In this situation, it is impossible to have $a_l<e_i\le c_1$ and $b_l<f_i\le d_1$ at the same time.}
        Suppose it happens for some $i$. Then the swap between $\{a_l,b_l\}$ and $\{e_i,f_i\}$ contains the edges $\{a_l,f_i\}$, $\{b_l,e_i\}$ and $\{c_1,d_1\}$, while the swap we are considering contains $\{a_l,d_1\}$, $\{b_l,c_1\}$ and $\{e_i,f_i\}$. The sum of weights for the former is greater than for the latter, contradicting that it was the leading term of $S$.

    If all the edges $\{e_i,f_i\}$ cross $\{a_l,d_1\}$, as they also cross $\{a_j,b_j\}$ for all $1\le j\le l-1$, they still form a $k+1$-crossing. The same happens if they cross $\{c_1,b_l\}$. It remains only to show that one of these two cases must hold. Suppose in the contrary that an edge $\{e_i,f_i\}$ does not cross $\{a_l,d_1\}$ and other edge $\{e_j,f_j\}$ does not cross $\{c_1,b_l\}$.
    
    Then either $d_1\le e_i<f_i\le a_l$ or $a_l\le e_i<f_i\le d_1$. The first possibility would imply that $\{e_i,f_i\}$ does not cross $\{a_l,b_l\}$, hence the second must hold. Then $a_l<e_i<b_1$ and $b_l<f_i\le d_1$. By the claim above, $a_l<e_i\le c_1$ is impossible, so $c_1<e_i<b_1$. By the same argument, we get $a_l<e_j\le c_1$ and $d_1<f_j<a_1$. Putting this together, we get
    \[a_l<e_j\le c_1<e_i<b_1<b_l<f_i\le d_1<f_j<a_1\]
    But this implies that $\{e_i,f_i\}$ and $\{e_j,f_j\}$ do not cross, a contradiction.
\end{proof}

This lemma immediately implies that Pfaffians are a Gr\"obner basis. More precisely:

\begin{theorem}
\label{thm:groebner}
Pfaffians of degree $2k+2$ are a Gr\"obner basis for $I_k(n)$ with respect to any fp-positive vector $v\in \Sep_n$. 

If the fp-positive vector is sufficiently generic then $\ini_v(I_k(n))$ is the monomial ideal generated by all $(k+1)$-crossings. 
\end{theorem}

The case $k=1$ of this theorem is classical, via the equality $\Pf{1}{n} = \Gr2n$, see \cite[Theorem 3.20]{PacStu} and Remark~\ref{rem:tree_metrics} below). In fact, in this case the last sentence in the theorem is an ``if and only if''. Indeed, $\Sep_n$ is, by definition, the closed Gr\"obner cone of $I_1(n)$ corresponding to the initial ideal generated by $2$-crossings. 

In general, let $\Grob_k(n)$ be the Gr\"obner cone of $I_k(n)$ corresponding to the ideal of $(k+1)$-crossings. For higher $k$ it is no longer true that $\Sep_n = \Grob_k(n)$, we only have the containement $\Sep_n\subset \Grob_k(n)$ which follows from the previous theorem.
Our next result explicitly describes $ \Grob_k(n)$.

A priori, for arbitrary $k$, the  Gr\"obner cone is described by the following family of linear inequalities, running over all the even cycles $(i_0,i_1\ldots,i_{2l-1},i_0)$
of length $2l$ that contain an $l$-crossing contained in a $(k+1)$-crossing, for $l\le k+1$:
\begin{align}
    v_{i_0i_1}-v_{i_1i_2}+\ldots-v_{i_{2l-1}i_0} \ge 0
\label{eq:cycles}
\end{align}

But most of these inequalities are redundant. For example, for $k=1$, $\Grob_1(n)=\Sep_n$ which is defined by $2\binom{n}4$ four-point conditions, but only the $\binom{n}2-n$ in Eqs.\eqref{eq:sep} are irredundant. In fact,  it turns out that for every $k$ and every $n\ge 2k+3$, the Gr\"obner cone is simplicial:

\begin{theorem}
\label{thm:groebner-cone}
For $n\ge 2k+3$, $\Grob_k(n)$ is, modulo the lineality space $L_n$, a simplicial cone given by the following inequalities, one for each $\{i,j\}$ with $|j-i|\ge 2$:
\begin{align}
    w_{ij}&\ge 0\quad \text{if } |j-i|>k  & \text{(long inequalities)}\label{eqrel} \\
    \sum_{i'\le i<j\le j'\le i'+k+1} w_{i'j'} & \ge 0\quad \text{if } 2\le |j-i|\le k & \text{(short inequalities)} \label{eqirrel}
\end{align}
The ray opposite to the facet indexed by $\{i,j\}$ is generated by:
\begin{itemize}
    \item The basis vector indexed by $\{i,j\}$ in the $w$ coordinates if $|j-i|\ge k+2$, and
    \item The negative basis vector indexed by $\{i+1,j\}$ in the $v$ coordinates if $|j-i| < k+2$.
\end{itemize}
\end{theorem}

Observe that the ``long'' inequalities are also facet-defining for $\Sep_n$ and the ``short'' ones are sums of facet-defining ``short'' inequalities in $\Sep_n$.

\begin{proof}
    First let's see that the inequalities are valid in the cone. The first group \eqref{eqrel} is obvious, because the $(k+1)$-crossing has higher weight than any swap. For the second, let $\{i,i+\ell\}$ be an edge with $\ell\le k$. For each edge $E\in T$ we call \emph{length of $E$ with respect to $U$} and denote it $\ell_U(E)$ the smallest size of the two parts of $U$ separated by $E$. For a matching $M$ of $U$ and an edge $E$ we denote by $c_M(E)$ the number of edges of $M$ that cross $E$.
    
    Consider the matching
    \[M=\{\{i+1,i+\ell\},\{i+2,i+k+3\},\{i+3,i+k+4\},\ldots,\{i+\ell-1,i+k+\ell\},\]
    \[\{i-k+\ell-1,i+\ell+1\},\{i-k+\ell,i+\ell+2\},\ldots,\{i,i+k+2\}\}\]
    This is a $k$-crossing plus the edge $\{i+1,i+\ell\}$. The coefficient of a $w$ will be the same in this matching than in the $(k+1)$-crossing, that is, $\ell_U(E)=c_M(E)$, except for the edges $\{i',j'\}$ with $i'\le i<i+l\le j'\le i'+k+1$, for which $c_M(E)=\ell_U(E)-2$. Hence the left hand side of \eqref{eqirrel} is half the difference between the weights, and the inequalities are true.
    
    Once we know that the inequalities are valid, let $G_{ij}$ be the ray defined in the statement. We only need to show that at each $G_{ij}$ all inequalities are equalities, except for the one of index $ij$, and that the $G_{ij}$ indeed lie in $\Grob_k(n)$.
    
    Indeed, if $|i-j|>k+1$ then $G_{ij}$ has all $w$ coordinates equal zero except $w_{ij}>0$. it is clear that all inequalities of the form \eqref{eqirrel} are equalities (since they only involve $w$s of length $\le k+1$) and all of type  \eqref{eqrel} except the one for $ij$ are equalities (by construction).
    If $|i-j|\le k+1$, in $G_{ij}$
    we have that the the only nonzero $v$ coordinate is $v_{i+1,j}$, which is negative. We take it equal to $-1$.
    Proposition \ref{prop:d_to_w} implies that in the $w$ coordinates the only non-zero ones are
    \[
    w_{i+1,j} = w_{i,j-1} =-\frac12,
    \qquad
    w_{i+1,j-1} = w_{i,j} =\frac12.
    \]
    
    Now, if $j-i\le k$, \eqref{eqrel} gives always 0 and \eqref{eqirrel} gives 1/2 exactly for one sum, the corresponding to $\{i,j\}$, and 0 for the rest. If $j-i=k+1$, \eqref{eqirrel} gives always 0 and \eqref{eqrel} gives 1 only for $w_{i,j}$.
    
    It remains to see that these rays are in $\Grob_k(n)$: 
    \begin{itemize}
        \item For the $w$ basis vectors this follows from the fact that they are in $\Sep_n$.
        \item For the negative $v$ basis vectors, we are giving weight $-1$ to an irrelevant edge and 0 to all the other edges; it is clear that every $(k+1)$-crossing gets weight zero, and every other matching gets nonpositive weight.
    \end{itemize}
\end{proof}

\begin{remark}\label{rem:cone}
For $n=2k+2$ the theorem fails, but it is also easy to describe $\Grob_k(n)$ since we have a single Pfaffian and the Gr\"obner fan is simply the normal cone of its Newton polytope. In particular, none of the equalities (\ref{eq:cycles})
is redundant and $\Grob_k(n)$ has as many facets as there are matchings of $[2k+2]$ whose symmetric difference with the $k+1$-crossing is a single cycle. For example:
\begin{itemize}
    \item For $k=2$, $n=6$, all matchings differ from the $3$-crossing in a single cycle. Thus, the $\Grob_2(6)$ has (modulo its lineality space) dimension $\binom{6}{2}-6 = 9$ and 14 facets.

    \item For $k=3$, $n=8$, there are matchings differing from the $4$-crossing in two cycles of length four. There are exactly 12 of them, coming from the three ways of partitioning the $4$-crossing into two pairs of edges and the two ways of completing each pair of edges into a four-cycle. Hence, $\Grob_3(8)$ has dimension $\binom{8}{2}-8 = 20$ and it has $105-1-12 = 92$ facets.
\end{itemize}

\end{remark}

Another difference between $k=1$ and $k>1$ is that for $k=1$ Pfaffians are a universal Gr\"obner basis for the ideal $I_1(n)$ (one proof is that every other Gr\"obner cone of $I_1(n)$ can be sent to $\Sep_n$ by a permutation of $[n]$, see \cite[Theorem 4.3]{SpeStu}). The same is known to fail for higher Grassmannians (see e.g., \cite[Section 7]{SpeStu} or \cite[Example 4.3.10]{MacStu}) and it also fails for higher Pfaffians:

\begin{example}[Pfaffians are not a universal Gr\"obner basis]
\label{exm:UGB}
Let $n=9$ and $k=3$. Consider the vector with
\begin{align*}
&v_{12}=v_{34}=v_{56}=v_{47}=v_{89}=2,\\
&v_{58}=v_{69}=1,\\
&v_{17}=v_{28}=v_{39}=10,
\end{align*}
and the rest of entries equal to zero. We are going to show that, regardless of the field $\K$, Pfaffians are not a Gr\"obner basis for this choice of $v$ (or any small perturbation of it).

Call $f$ and $g$ the Pfaffians on the sets $U=\{1,2,3,4,5,6\}$ and 
$V=\{4,5,6,7,$ $8,9\}$, which have as matchings of highest weight $\{12,34,56\}$ and $\{56,47,89\}$, both of weight 6. That is, 
\[
in(f)=x_{12}x_{34}x_{56},\qquad
in(g)=x_{56}x_{47}x_{89}.
\]
The following polynomial, which is nothing but the $S$-polynomial of $f$ and $g$ that arises in Buchberger's algorithm, lies in $I_3(9)$
\[
h:=x_{12}x_{34}\, g - x_{47}x_{89} \,f.
\]

The only monomials of weight $> 6$  in $h$ are the initial terms of the two parts
$x_{12}x_{34}\, g$ and $ x_{47}x_{89}\, f$, which cancel out, and $x_{12}x_{34} \,x_{47}x_{58}x_{69}$, of weight 8. Hence, we have that
$in(h) =x_{12}x_{34} \,x_{47}x_{58}x_{69}$.

In particular, if  Pfaffians were a universal Gr\"obner basis, there should be a Pfaffian  whose leading monomial divides $in (h)$. That is, there should be a set $W\subset [9]$ of six elements whose matching $M$ of maximum weight is contained in $\{12, {34} ,{47},{58}, 69\}$.
This $W$ does not exist. Indeed, $W$ cannot contain any of the pairs $\{1,7\}$, $\{2,8\}$ or $\{3,9\}$, because then its highest matching would have weight $\ge 10$. And every set of three edges among 
$\{12, {34} ,{47},{58}, 69\}$ not containing any of those pairs of vertices contains the edges $\{{58}, 69\}$, which cannot be in the leading term of any Pfaffian since they produce smaller weight than their swap $\{{56}, 89\}$.
\end{example}

\begin{remark}
\label{rem:groebnerpfaffians}
That Pfaffians are a Gr\"obner basis for the ideal $I_k(n)$ they generate is known since long. The earliest proof we are aware of is by Herzog and Trung~\cite{HerTru}, who construct a lexicographic order for which the initial ideal $in_<(I_k(n))$ is generated by the $(k+1)$-nestings. Here $\{a,d\}$ and $\{b,c\}$ are nested if $1\le a<b<c<d\le n$. 

This result is recovered by Sturmfels and Sullivant~\cite{StuSul} as a special case of a more general behaviour; Sturmfels and Sullivant study the relation between the Gr\"obner bases of an ideal $I$ and those of its secant ideals $I^{\{k\}}$, and call a monomial order ``delightful'' if the initial ideal of $I^{\{k\}}$ can be obtained from that of $I$ by the following simple combinatorial rule: the standard monomials in $in_<(I^{\{k\}})$ are the products of $k$ standard monomials of $in_<(I)$. They then consider $I_k(n) = I_1(n)^{\{k\}}$ as an example~\cite[Example 4.13]{StuSul}, and show that the lexicographic order of Herzog and Trung~\cite{HerTru} is delightful.

Much closer to our framework is the work of Jonsson and Welker~\cite{JonWel}. Taking a lexicographic order different from that of \cite{HerTru} they obtain as initial ideal the same one from our Theorem~\ref{thm:groebner}, generated by $(k+1)$-crossings. (This same result is stated without proof in \cite[p.~107]{PacStu}).
Our Theorem~\ref{thm:groebner} is a bit more general, in that it says that \emph{any} order that produces for $I_1(n)$ the initial ideal generated by $2$-crossings automatically produces for $I_k(n)$ the initial ideal generated by $(k+1)$-crossings. That is, we show the fp-positive cone, which is a Gr\"obner cone of $I_1(n)$ by definition, to be contained in a Gr\"obner cone of $I_k(n)$ for every $k$.

It is worth noticing that our orders are not ``delightful'' in the sense of~\cite{StuSul}. Indeed, the maximal square-free standard monomials in our initial ideal are the $k$-triangulations of the $n$-gon, and not every $k$-triangulation is the union of $k$ triangulations. For a trivial example observe that the complete graph on $5$ vertices is a $2$-triangulation but it is not the union of two triangulations of the pentagon. Related to this, see \cite[Section 6]{PilSan}.
\end{remark}

Theorem~\ref{thm:groebner} has a natural interpretation via $(k+1)$-free sets and multitriangulations. Observe that $k(2n-2k-1)$, the dimension of $\Ass{k}{n}$, coincides with that of $\Pf{k}{n}$.

\begin{corollary}
    \label{coro:groebner}
    If the weight vector $v$ for the variables in $\K[x_{i,j}, \{i,j\}\in \bnn]$ lies in $\Grob_k(n)$ (in particular, if it is fp-positive)  and generic then the initial ideal of $I_k(n)$ equals the Stanley-Reisner ideal of the extended $k$-associahedron $\Ass{k}{n}$. That is: it is the radical monomial ideal whose square-free standard monomials form, as a simplicial complex, $\Ass{k}{n}$.
\end{corollary}

\subsection{The algebraic matroid of $\Pf{k}{n}$ and low-rank matrix completion}
\label{sec:matroid}

Let $I \subset \K[x_1,\dots,x_N]$ be a prime ideal, the \emph{algebraic matroid} of $I$, which we denote $\MM(I)$ has the variables $E:=\{x_1,\dots,x_N\}$ as elements and a subset $S\subset E$ is independent if $I$ does not contain any non-zero polynomial in  $\K[S]$. If $\K$ is algebraically closed and $V=V(I)$ is the irreducible variety of $V$, then dependence and independence of a subset $S$ of variables can be told via the natural projection map $\pi_S: V\subset \K^N \to \K^S$, as follows.

a set is independent in $\MM(I)$ if, and only if, the corresponding projection map $\pi_S:V\to K^S$ is dominant; that is, its image is (Zariski) dense.
We use \cite{Rosen, RST:algebraic, KRT} as our main sources for algebraic matroids.

\begin{theorem}
\label{thm:algebraic}
Let $\K$ be an algebraically closed field, $I\subset \K[x_1,\dots,x_N]$ a prime ideal and $V$ its algebraic variety. For each $S\subset [N]$
denote by $\pi_S : \K^{[N]}\to \K^S$ the coordinate projection to $S$. Then:
\begin{enumerate}
\item $S$ is independent in $\MM(I)$ if and only if $\pi_S(V)$ is  Zariski dense in $\K^S$.
\item The rank of $S$ is equal to the dimension of $\pi_S(V)$.
\item $S$ is spanning if and only if  $\pi_S$ is finite-to-one: for every $x\in \K^S$ the fiber $\pi_S^{-1}(x)$ is finite (perhaps empty).
\end{enumerate}
\end{theorem}

\begin{proof}
	The first part is Theorem 15 in \cite{RST:algebraic}.
 For the second, the rank of $S$ is the maximum size among independent subsets of $S$, which are the subsets $T$ for which  $\pi_T(V)=\pi_T(\pi_S(V))$ has dimension $|T|$. The maximal ones are those which have the same size as the dimension of $\pi_S(V)$, so this is the rank.

	The third part is a consequence of the second, because a projection has the same dimension than the variety if and only if the fiber has dimension zero, and a fiber has dimension zero if and only if it is finite.
\end{proof}

This statement has as a consequence that, over an algebraically closed field, we can speak of the algebraic matroid of the irreducible variety $V$, and denote it $\MM(V)$, instead of looking at the ideal.

We now turn our attention to the case of $\Pf{k}{n}$.
\begin{corollary}
	\label{coro:bases}
	$(k+1)$-free subsets of edges are independent in the algebraic matroid of $\Pf{k}{n}$ and $k$-triangulations are bases.
\end{corollary}

\begin{proof}
	Let $S$ be a dependent set in the matroid. Then there is a polynomial $f$ in $I_k(n)$ using only the variables in $S$ and the initial monomial of $f$ according to any fp-positive weight also uses only variables in $S$. By Corollary~\ref{coro:groebner} $I_k(n)$ has an initial ideal consisting only of $(k+1)$-crossing monomials, hence $f$ has a monomial with a $(k+1)$-crossing, and  $S$ is not $(k+1)$-free.
	
	For the second part, it is enough to see that the rank of the matroid equals $2nk-\binom{2k+1}2$. This is because points in $\Pf{k}{n}$ are antisymmetric matrices of rank $\le 2k$. In order to construct one such matrix $M$ we can choose generic elements in the first $2k$ rows above the diagonal and every other element $M_{i,j}$  ($i,j>2k$) is uniquely determined by them. Indeed, the Pfaffian of the rows and columns indexed by $[2k]\cup\{i,j\}$ has the form $A M_{i,j} + B$ where $A$ is the Pfaffian of $[2k]$. Since our choice was generic, $A\ne 0$.
\end{proof}

This proof already shows the relation between independence in the algebraic matroid of $\Pf{k}{n}$ and low-rank completion of partially known antisymmetric matrices. Suppose that we are given a matrix $M\in \K^{n\times n}$, of which we only know a subset $T$ of entries, we want to deduce the rest of entries with the restriction that $M$ needs to be antisymmetric and have at most range $2k$.
Corollary~\ref{coro:bases} then immediately allows us to prove Theorem~\ref{thm:matroid}:


\begin{proof}[Proof of Theorem~\ref{thm:matroid}]
Consider the projection $\pi_T:\K^\bnn\to \K^T$ that keeps only the coordinates of $T$. In part (1) we are saying that $\pi_T$ is almost surjective (any element has a preimage except for a zero measure set) and in part (2) that it is finite-to-one (every point $x\in \K^T$ has a finite fiber $\pi^-1(x)$).
Both parts follow  from Corollary~\ref{coro:bases}, via the characterization of algebraic matroids in Theorem~\ref{thm:algebraic}
%
%
\end{proof}

It is worth mentioning that the algebraic matroid of $\Pf{k}{n}$ coincides with the \emph{generic hyperconnectivity matroid in dimension $2k$} introduced by Kalai~\cite{Kalai}. Let us review this relation.

The \emph{hyperconnectivity matrix} of a configuration $\{\p_1,\dots,\p_n\} \subset \R^d$ is defined to be
\[
H(\p):=
\begin{pmatrix}
\p_2 & -\p_1 & 0 & \dots & 0 & 0 \\
\p_3 & 0 & -\p_1 & \dots & 0 & 0 \\
\vdots & \vdots & \vdots & & \vdots & \vdots \\
\p_n & 0 & 0 & \dots & 0 & -\p_1 \\
0 & \p_3 & -\p_2 & \dots & 0 & 0 \\
\vdots & \vdots & \vdots & & \vdots & \vdots \\
0 & 0 & 0 & \dots & \p_n & -\p_{n-1}
\end{pmatrix}.
\]
We call \emph{hyperconnectivity matroid of $\p$} the  linear matroid $\HH_d(\p)$ of rows of $H(\p)$. There clearly exists an open dense subset of configurations where the matroid is the most free; we call that matroid the \emph{generic hyperconnectivity matroid} of dimension $d$ and denote it $\HH_d$.

On the other hand, if an algebraic variety $V$ is parametrized by a polynomial map $T:\R^M \to V\subset \R^N$, then the algebraic matroid of $V$ equals the linear matroid of rows of the Jacobian of $T$ at a sufficiently generic point of $\R^M$~\cite[Proposition 2.5]{Rosen}.
In our case, $\Pf{k}{n}$ is parametrized by the following linear map:
\begin{align}
T: (\R^n)^{2k} &\to \Pf{k}{n} \subset \R^{\binom{n}2} \nonumber\\
(\ab_1,\bb_1,\dots,\ab_k,\bb_k) &\mapsto \sum_{l=1}^k \left(a_{l,i}b_{l,j}-a_{l,j}b_{l,i}\right)_{1\le i<j\le n},
\label{eq:T}
\end{align}
where $\ab_l=(a_{l,1},\dots,a_{l,n})$ and $\bb_l=(b_{l,1},\dots,b_{l,n})$.
The Jacobian of $T$ at a point $(\ab_1,\bb_1,\dots,\ab_k,\bb_k)$ then coincides with the hyperconnectivity matrix of the configuration 
$(\p_1,\dots,\p_n)$ where
\[
\p_i=(b_{1,i}, -a_{1,i},\dots, b_{k,i}, -a_{k,i}).
\]

In particular:

\begin{proposition}
\label{prop:H_equals_S}
The algebraic matroid of $\Pf{k}{n}$ coincides with the generic hyperconnectivity matoid in dimension $2k$.
\end{proposition}

\begin{corollary}
\label{cor:hyperconnectivity}
$k$-triangulations are bases in the generic hyperconnectivity matroid of dimension $2k$.
\end{corollary}

This statement is related to the following conjecture of Pilaud and Santos:

\begin{conjecture}[\protect{\cite[Conjecture 8.6]{PilSan}}]
\label{cor:bar-and-joint}
$k$-triangulations are bases in the generic bar-and-joint rigidity matroid of dimension $2k$.
\end{conjecture}


Let us denote $\RR_d$ the generic bar-and-joint rigidity matroid.
It is known that hyperconnectivity falls under the framework of rigidity theory in the sense that both $\RR_d$ and $\HH_d$ are \emph{abstract rigidity matroid} as defined by Edmonds; matroids of rank $dn-\binom{d+1}{2}$ on the ground set $\bnn$ with the property that every complete graph on $d+1$ elements is independent. It is conjectured that $\RR_d$ is freer than $\HH$, which would make Proposition~\ref{cor:hyperconnectivity} imply Conjecture~\ref{cor:bar-and-joint}, but the conjecture is open starting at dimension $3$. (For $d=1$ both matroids coincide with the usual graphical matroid of the complete graph; for $d=2$ there are combinatorial characterizations of independent graphs in both: Laman graphs in $\RR_2$, and the graphs described in \cite{Bernstein} in $\HH_2$).

It is known, however, that for points chosen along the moment curve the two matroids coincide:

\begin{theorem}[\cite{CreSan}]
\label{thm:moment_curve}
	Let $d$ be a positive integer and let $t_1,\ldots,t_n\in \R$ be (distinct) real numbers. Let $\p=(\p_1,\dots,\p_n)\subset \R^d$ be the corresponding configuration of points along the moment curve, so that  $\p_i=(t_i,\dots,t_i^{d})$, $i=1,\dots,n$. Then,
	 $\HH_d(\p) = \RR_d(\p)$.
	\end{theorem}

In particular, a statement that would imply both Proposition~\ref{cor:hyperconnectivity} and Conjecture~\ref{cor:bar-and-joint} is: $k$-triangulations are bases for the matroid $\HH_d(\p) = \RR_d(\p)$ when $\p$ is a generic collection of points along the moment curve. We refer to~\cite{CreSan} and the references in there for an up-to-date account of the relation between $\HH_d$ and $\RR_d$.

\section{The tropicalization of $\Pf{k}{n}$}
\label{sec:Vn}

\subsection{The tropical Pfaffian  variety and prevariety}

Recall that the \emph{tropical hypersurface} $\trop(f)$ of a polynomial $f\in \K[x_1,\dots,x_N]$ is the collection of weight vectors $v\in \R^{N}$ for which $\ini_v(f)$ is not a monomial. 
Put differently, the weight vectors for which the maximum weight among monomials in $f$ is attained at least twice. 
It is a polyhedral fan, namely the codimension one skeleton of the normal fan of the Newton polytope of $f$. 

If $V$ is the algebraic variety of an ideal $I$, the
\emph{tropicalization} of $V$ equals
\[
\trop(V) := \cap_{f\in I} \trop(f).
\]
A finite subset $B\subset I$ such that $\trop(V) := \cap_{f\in B} \trop(f)$, which always exists, is called a \emph{tropical basis of $I$}. Not every generating set of $I$ (not even a universal Gr\"obner basis of $I$, see~\cite[Example 10]{BJSST} or~\cite[Example 2.6.1]{MacStu}) is a tropical basis. In general, a finite intersection of tropical hypersurfaces is called a \emph{tropical prevariety}, while the tropicalization of a variety is a \emph{tropical variety}~\cite[Definitions 3.1.1 and 3.2.1]{MacStu}.
The tropical variety defined by a finite set of polynomials $\{f_1,\dots,f_n\}$ contains, but is sometimes not equal to, the tropical variety of the ideal $(f_1,\dots,f_n)$ generated by them.

Looking at the case of Pfaffians,
for each subset $U$ of $[n]$ of size $2k+2$ we have as tropical hypersurface the set of vectors $v\in \R^{\bnn}$ for which the maximum
\[
\Big\{\sum_{\{i,j\}\in E}v_{ij}:E\text{ matching in }U\Big\},
\]
is attained at least twice.
We denote by $\PV{k}{n}$ the intersection of all these tropical hypersurfaces for the different $U\in \binom{[n]}{2k}$. We call it the \emph{tropical Pfaffian prevariety}. It 
contains the tropicalization $\trop(\Pf{k}{n})$ of $\Pf{k}{n}$ and it is known to coincide with it in the following cases:

\begin{itemize}
\item If $n=2k+2$, since then we have a single Pfaffian defining $\trop(\Pf{k}{n})$.
\item If $k=1$, by the results in \cite{SpeStu} and the fact that $\Pf{1}{n}$ coincides with the Grassmannian $\Gr2n$ (see Remark~\ref{rem:tree_metrics} below).
\end{itemize}

The following example looks at the first open case:

\begin{example}
For $k=2$ and $n=7$, using {\tt Gfan} \cite{Gfan} we have computed $\PV{2}{7}$ as the intersection of the seven hypersurfaces corresponding to Pfaffians. The result is a non-simplicial fan of pure dimension 18 with 77 rays and a lineality space of dimension 7 (as expected). It has 73395 maximal cones, all of them with multiplicity 1. These cones correspond to 33 classes of symmetry via permutations of variables.
The 77 rays are:
\begin{itemize}
    \item The 21 vectors in the standard basis of the coordinates $v$, and their 21 opposites. That is, for each $\{i,j\}\in \binom72$, the two vectors with $v_{ij} = \pm 1$ and  $v_{i'j'}=0$ otherwise. 
    \item The 35 vectors obtained as follows: for each $\{i,j,k\}\in \binom73$, the vector with $v_{ij} =v_{ik} =v_{jk} = 1$ and $v_{i'j'}=0$ otherwise. 
    \end{itemize}
7 of the 14 extremal rays of $\Sep_7$ are among these vectors. In the $w$ coordinates these are the vectors with $w_{ij} =1$ and all other entries equal to zero, for the fourteen choices of non-consecutive $i$ and $j$. The seven with $i=j-2$ coincide (modulo the lineality space) with the $v$-basis vectors with $v_{j-1,j} = - 1$, which are rays, and the seven with $i=j-3$ are the vectors with $v_{j-2,j} =v_{j-1,j} =v_{j-2,j-1} = - 1$, that is, the opposites to some rays, but they are not rays themselves.
None of the other 77 rays computed by {\tt Gfan} lie in $\Sep_7$.

The cone corresponding to a given 2-triangulation cannot be in this prevariety, because its rays are not among those rays. But it can be the result of intersecting a cone from the prevariety with $\Sep_7$, because, by Remark \ref{rem:cone}, the Gr\"obner cone in which it is contained is a bit greater than $\Sep_7$. In fact, a $v$ coming from a 2-triangulation is in the cone spanned by the rays $v_{j-1,j}=-1$ for all $j$ and $v_{j-2,j}=-1$ for $\{j-3,j\}$ in the 2-triangulation. The intersection of this cone with $\Sep_7$ is the cone in the 2-associahedron.

In this case, we want to check whether the tropical prevariety $\PV{2}{7}$ coincides with the variety $\trop(\Pf{2}{7})$. To do that, we need to compute the tropical variety as a subfan of the Gröbner fan. However, it is not enough to check that the cones in both fans are the same, because the tropical prevariety may not be a subfan of the Gröbner fan.

$\trop(\Pf{2}{7})$ is a simplicial fan with 84420 cones, that belong to 35 equivalence classes. The equality as sets for the two fans can now be checked by showing that all the simplicial cones in $\trop(\Pf{2}{7})$ are contained in a cone of $\PV{2}{7}$ and the union of the cones contained in the same cone gives the whole cone.

The prevariety contains 71820 simplicial and 1575 non-simplicial cones. The simplicial ones are also cones of the variety, so that part is correct. Now there are 12600 remaining cones in the variety, that correspond to the non-simplicial part. The non-simplicial cones can be triangulated in two ways: in 8 cones and in 3 cones. The triangulation in 8 cones of all them can be shown to match exactly the cones of the variety, and we are done.
\end{example}

To better understand the difference between $\PV{k}{n}$ and $\trop(\Pf{k}{n})$ we are now going to relate them to two different notions of rank for a tropical matrix. For this, it is convenient to extend $\R$ to the \emph{tropical semiring} $\T:= \R\cup\{-\infty\}$, with the operations  $\max$ as ``addition'' and $+$ as ``multiplication''. By a tropical $n_1\times n_2$-matrix we mean an $n_1\times n_2$-matrix with entries in $\T$. To distinguish between tropical (pre)-varieties in $\R^n$ and $\T^n$ we denote $\overline V$ the extension to $\T^n$ of a tropical variety or prevariety $V\in \R^n$. 

Clearly, for every family $F$ of polynomials, the prevariety of $F$ in $\T^n$ is topologically closed, so it contains the closure of the prevariety in $\R^n$, and the same holds for varieties. The converse is not always true, as the following example shows:

\begin{example}
Let $I= (x_1 x_3+x_2, x_2 x_3+x_1)$. 
The tropical variety it defines in $\R^3$ equals $\{(a,a,0) : a \in \R\}$, while the variety it defines in $\T^3$ contains that plus the points $\{(-\infty,-\infty,b) : b \in \T\}$. 

Observe that this ideal is not prime, since it contains $x_1(x_3^2-1)$ but it does not contain any of its factors $x_1$, $x_3 + 1$ or $x_3 - 1$. We do not know whether for prime ideals it is always true that the closure of $V$ equals $\overline V$.
\end{example}


The following two  notions of  rank were  introduced in \cite{DSS}.

\begin{definition}[Tropical rank, \protect{\cite[Def. 5.3.3]{MacStu}}]
	A square matrix $M\in \T^{r\times r}$ is \emph{tropically singular} if the maximum in the tropical determinant
	\[
	\trop\det(M):=\max_{\sigma\in S_r}\sum_{i=1}^r m_{i\sigma(i)}
	\]
	is attained at least twice, and \emph{tropically regular} otherwise.

	The \emph{tropical rank} of a tropical matrix is the largest size of a tropically regular minor in it.
\end{definition}

Stated differently, the tropical rank of $M$ is the largest $r$ such that $M$ is not in the tropical prevariety of the $r\times r$ minors or, equivalently, the smallest $r$ such that $M$ is in the tropical prevariety of the $(r+1)\times (r+1)$ minors.

\begin{definition}[Kapranov rank, \protect{\cite[Def. 5.3.2]{MacStu}}]
Let $M\in \T^{n_1\times n_2}$ be a tropical matrix.
	The \emph{Kapranov rank} of $M$ over a valuated field $\K$ is the smallest rank of a lift of the matrix, that is, a matrix $\widetilde{M}\in\K$ such that the degree of $\widetilde{M}_{ij}$ is $m_{ij}$.
\end{definition}

The tropical variety of the $(r+1)\times(r+1)$ minors is the tropicalization of the (classical) variety of the matrices with rank at most $r$. Hence, the Kapranov rank is the smallest $r$ such that $M$ is in the tropical variety of the $(r+1)\times (r+1)$ minors, or the largest $r$ such that $M$ is not in the tropical variety of the $r\times r$ minors. 

Observe that the Kapranov rank of $M$ depends on the field $\K$ under consideration, while the tropical rank does not. The relation of the two notions of rank to the tropical variety and prevariety of minors readily shows that the Kapranov rank is greater or equal than the tropical rank \cite[Theorem 1.4]{DSS}. Two small examples where the two notions do not coincide appear in \cite[Section 7]{DSS} (a $7\times 7$ matrix of tropical rank three and Kapranov rank four) and  \cite{Shitov0} (a $6\times 6$ matrix of tropical rank four and Kapranov rank five). The two examples are reproduced in \cite[Section 4]{Shitov2} where Shitov, completing work of Develin-Santos-Sturmfels \cite{DSS}, Chan-Jensen-Rubei \cite{CJR}, and himself \cite{Shitov1} shows that these two examples are the smallest possible:

\begin{lemma}[\cite{Shitov2}]
For given positive integers $r,n_1,n_2$ the following are equivalent:
\begin{enumerate}
\item The $(r+1)\times (r+1)$ minors are a tropical basis for the variety of $n_1\times n_2$ matrices of rank $r$ (over any of the complex, real, or rational fields).
\item $r\le 2$, or $r=\min\{n_1,n_2\}$, or \ $r=3$ and $\min\{n_1,n_2\}\le 6$.
\end{enumerate}
\end{lemma}

Since these notions of rank distinguish between the variety and prevariety of minors, antisymmetric versions of them will distinguish between the variety and prevariety of Pfaffians. (The same idea for the \emph{symmetric} case is explored in \cite{Zwick}).

	Let $M\in \R^{n_1\times n_2}$ be a tropical matrix and let $n=n_1+n_2$. Let $K\in \R$ be a sufficiently big constant. From $M$ and $K$ we construct the following $n\times n$ matrix:
	\[
	\Skew(M,K) :=
	\left(
	\begin{matrix}
		N_1 & M \\ M^t & N_2
	\end{matrix}
	\right)
	\in \T^{n\times n},
	\]
	where $(N_1)_{ij}=m_{i1}+m_{j1}-K$ and $(N_2)_{ij}=m_{1i}+m_{1j}-K$ for $i\ne j$, and $(N_1)_{ii}=(N_2)_{ii}=-\infty$.
	We have a corresponding vector $v(M,K) \in \R^{\bnn}$ of entries of $\Skew(M,K)$:
	\[
	v_{ij}:=
	\begin{cases}
		m_{i,j-n_1} &\text{ if } 1\le i \le n_1 < j. \\
		m_{i1}+m_{j1}-K &\text{ if } 1\le i,j \le n_1. \\
		m_{1,i-n_1}+m_{1,j-n_1}-K &\text{ if } i,j>n_1.
	\end{cases}
	\]
	We also consider the matrix and vector $\Skew(M,\infty)$ and $v(M,\infty) \in \T^{\bnn}$ obtained using $\infty$ instead of $K$. That is:
		\[
	\Skew(M,\infty) :=
	\left(
	\begin{matrix}
		-\infty & M \\ M^t & -\infty
	\end{matrix}
	\right)
	\in \T^{n\times n},
	\]

	
	\begin{lemma}
	\label{lemma:ranks}
		Let $M\in \T^{n_1\times n_2}$ be a tropical matrix and $K\in \overline \R$. For the vector $v(M,K)\in \T^\bnn$  defined above we have:
		\begin{enumerate}
			\item For  $K$ sufficiently large, $v(M,K) \in \PV{k}{n}$ if and only if the tropical rank of $M$ is at most $k$.
			\item $v(M,\infty) \in  \trop(\Pf{k}{n})$  if and only if  the Kapranov rank of $M$ is at most $k$.
		\end{enumerate}
	\end{lemma}

\begin{proof}
	For part (1), assume first that  $v(M,K) \in \PV{k}{n}$, and consider a $(k+1)\times (k+1)$ minor of $M$. This corresponds to a set  $U\in \binom{[n]}{2k+2}$ with half of the elements in $[1,\dots, n_1]$ and the other half in  $[n_1+1,\dots, n]$. Since  $v(M) \in \PV{k}{n}$, there are at least two perfect matchings in $U$ of maximum weight. Since we chose $K$ very big, none of these matchings come from the $N_1$ or $N_2$ parts of  $\Skew(M,K)$. This implies that the minor of $M$ that we started with is tropically singular.

	Conversely, assume that $\trop \rank M \le k$. Let $U\in \binom{[n]}{2k+2}$ and consider a perfect matching $E$ in $U$ with maximal weight, which is a term in the Pfaffian of $U$. We have three cases:
	\begin{itemize}
		\item If all the edges in $E$ are between $[n_1]$ and $[n_1+1,n]$, $E$ corresponds to a permutation in $M$ attaining the tropical determinant. As $\trop \rank M \le k$, there must be another permutation with the same weight.
		
		\item If all the edges in $E$ except one are between $[n_1]$ and $[n_1+1,n]$, suppose $E=\{\{i_1,j_1\},\dots,\{i_{k+1},j_{k+1}\}\}$, and $i_1,\dots,i_{k+1},j_1\le n_1<j_2,\dots,j_{k+1}$ (the other case is symmetric). Then
		\[
		w(E)=v_{i_1j_1}+\dots+v_{i_{k+1}j_{k+1}}=m_{i_11}+m_{j_11}+m_{i_2,j_2-n_1}+\ldots+m_{i_{k+1},j_{k+1}-n_1}-K.
		\]
		
		We have now two cases:
		\begin{itemize}
			\item If $j_l=n_1+1$ for some $l$, for example $j_2=n_1+1$, then
			\[
			w(E)=v_{i_1i_2}+v_{j_1,n_1+1}+v_{i_3j_3}+\ldots+v_{i_{k+1}j_{k+1}}.
			\]
			
			\item If $j_l>n_1+1$ for all $l$, $w(E)-m_{j_11}+K$ is the weight of the permutation $\{i_1,1\}, \{i_2,j_2-n_1\},\dots,\{i_{k+1},j_{k+1}-n_1\}$ in $M$. Since the tropical rank of $M$ is smaller than $k+1$, there is another permutation 
			with weight greater or equal than $w(E)-m_{j_11}+K$. That is,
			\[
			w(E)-m_{j_11}+K\le m_{i_1'1}+m_{i_2',j_2-n_1}+\ldots+m_{i_{k+1}',j_{k+1}-n_1}
			\]
			where $(i_1',\dots,i_{k+1}')$ is a permutation of $(i_1,\dots,i_{k+1})$. Equivalently
			\begin{align*}
			w(E)\le &(m_{i_1'1}+m_{j_11}-K) + m_{i_2',j_2-n_1}+\ldots+m_{i_{k+1}',j_{k+1}-n_1} =
			\\
			=& v_{i_1'j_1}+v_{i_2'j_2}+\ldots+v_{i_{k+1}'j_{k+1}}.
			\end{align*}

			As $E$ is maximal, this is an equality, and we have another matching in $U$ with the same weight.
		\end{itemize}
	
		\item If there is more than one edge inside $[n_1]$ or inside $[n_1+1,n]$, suppose for example we have the edges $\{a,b\}$ and $\{c,d\}$ with $a,b,c,d\le n_1$. Then any of the two swaps among these four elements preserves weight, indeed:
		\[
		v_{a,b} + v_{c,d} =m_{a1}+m_{b1}+m_{c1} + m_{d1} -2K= v_{a,c} + v_{b,d} = v_{a,d} + v_{b,c}
		\]		
	\end{itemize}
	
	In any case, there is another matching with the same weight as $E$, and this finishes part (1).
	
	For part (2),  if $M$ has Kapranov rank at most $k$ then there is a lift $\widetilde{M}$ of $M$ of rank $k$. Thus,
	\[
	\left(
	\begin{matrix}
		0 & \widetilde M \\ \widetilde M^t & 0
	\end{matrix}
	\right)
	\]
	is an antisymmetric  lift of $\Skew(M,\infty)$ of rank $2k$.
	
	Conversely,
	 if $v(M,\infty)\in  \trop(\Pf{k}{n})$, consider  an antisymmetric matrix in $\Pf{k}{n}$
	 projecting to it, hence of rank $2k$. This matrix necessarily has zero entries in the places where $v(M,\infty)$ has $-\infty$, so it is of the form 
	 	\[
	\left(
	\begin{matrix}
		0 & \widetilde M \\ \widetilde M^t & 0
	\end{matrix}
	\right),
	\]
	where $\widetilde{M}$   is a matrix of rank at most $k$ and projecting to $M$.
\end{proof}

\begin{theorem}
	\label{thm:ranks}
If there is a matrix $M\in \R^{n_1\times n_2}$ of tropical rank $\le k$ and Kapranov rank $>k$ then 
$\PV{k}{n}\ne \trop(\Pf{k}{n})$, where $n=n_1+n_2$.

This happens, for example, for $k=3$ and any $n\ge 14$ and for any $k\ge 4$ and $n\ge 2k+4$.
%
\end{theorem}

\begin{proof}
Let $M\in \R^{n_1\times n_2}$ be a matrix of tropical rank $\le k$ and Kapranov rank $>k$. By Part (1) of Lemma~\ref{lemma:ranks} we have that $v(M,K) \in \PV{k}{n}$ for every sufficiently big $K$.

 Also, by Part (2) of the Lemma, $v(M,\infty) \not\in \overline{\trop (\Pf{k}{n})}$. 
 In particular, $v(M,\infty)$ is not in the closure of ${\trop (\Pf{k}{n})}$, 
 which implies it is not true that $v(M,K) \in \trop (\Pf{k}{n})$ for all sufficiently big $K$. 
 
 Thus, $\PV{k}{n} \ne  \trop (\Pf{k}{n})$.
\end{proof}

Summing up, the cases where we do not know whether $\PV{k}{n}= \trop(\Pf{k}{n})$ are:
\begin{itemize}
\item $k=2$ and $n\ge 8$,
\item $k=3$ and $n\in \{9,10,11,12,13\}$,
\item $k\ge4$ and $n = 2k+3$.
\end{itemize}

\subsection{The $k$-associahedron as the fp-positive part of the tropical Pfaffian  variety}

We are interested in the part of $\PV{k}{n}$ contained in $\Grob_k(n)$:

\begin{definition}
\label{defi:positive}
We define
\[
\PVplus{k}{n} := 
\PV{k}{n}\cap \Grob_k(n).
\]
We  call it the \emph{$(k+1)$-free part} of the tropical Pfaffian variety of parameters $n$ and $k$ for two reasons. On the one hand, the initial ideal corresponding to $\Grob_k(n)$ is the Stanley-Reisner ideal of the complex of $(k+1)$-free sets. But, more significantly, our results in this section say that $\PVplus{k}{n}$ coincides with the points of $\Grob_{k}(n)$ which, expressed in the $w$-coordinates, have $(k+1)$-free support.
\end{definition}

\begin{theorem}\label{thm:main}
Let $v=d(w)\in \Grob_k(n)$ be a vector in the Gr\"obner cone. This includes the case where $w$ is non-negative (or, equivalently, $v\in\Sep_n$).
    Then,
    \begin{enumerate}
        \item 
	$v\in\PVplus{k}{n}$ if and only if the support of $w$ is $(k+1)$-free. 
	
	\item If the above holds, then for every subset $U\subset \binom{[n]}2$ of size $2k+2$ one of the maximal matchings of $U$ for $v$ is the one producing a $(k+1)$-crossing, and a second one is obtained from it by a swap of two consecutive edges in the $(k+1)$-crossing.
    \end{enumerate}
\end{theorem}
\begin{proof}

	Let $U=\{a_0,a_1,\ldots,a_{2k+1}\}$ written in cyclic order, and let $E_0$ be the $(k+1)$-crossing in it, that is, the matching that pairs $a_i$ with $a_{k+1+i}$. As we already know, the maximum weight given by $v$ to matchings of $U$ is attained at $E_0$.
	
	If the support of $w$ is $(k+1)$-free, there must be an $l$ such that no edge in the support of $w$ has an end between sides $a_l$ and $a_{l+1}$ and the other between $a_{l+k+1}$ and $a_{l+k+2}$. Then,  let $E_1=E_0\setminus\{\{a_l,a_{l+k+1}\},\{a_{l+1},a_{l+k+2}\}\}\cup\{\{a_l,a_{l+k+2}\},\{a_{l+1},a_{l+k+1}\}\}$ has the same weight as $E_0$, so that $v\in \PVplus{k}{n}$ and part (2) holds.
	
	Conversely, if the support of $w$ contains a $(k+1)$-crossing then there is a $U=\{a_0,a_1,\ldots,a_{2k+1}\}$ such that each $a_i$ lies in one of the $2k+2$ regions defined by that crossing, and then the matching $E_0$ of $U$ has weight strictly larger than any other matching. In particular, $v\not\in \PV{k}{n}$.
\end{proof}

We now want to show that $\PVplus{k}{n}$ is contained in $\trop(\Pf{k}{n})$. That is to say, even if the tropical Pfaffian variety and prevariety may not coincide, their ``$(k+1)$-free parts'' coincide. We need the following Lemma, the proof of which we postpone to Section~\ref{sec:balanced}:

\begin{lemma}
\label{lemma:balanced}
Let $v=d(w)\in \Grob_k(n)$ be sufficiently generic. Then, for every subset $U\in \binom{[n]}{2k+2}$ we have that $U$ has the same number of positive and negative matchings of maximum weight with respect to $v$.
\end{lemma}

\begin{corollary}
    \label{cor:balanced}
    $\PVplus{k}{n}  \subset \trop(\Pf{k}{n})$. Moreover, $\PVplus{k}{n} \subset \trop^+(\Pf{k}{n})$.
\end{corollary}

Let us point out that $\PV{k}{n}$ and $\PVplus{k}{n}$ are independent of the field $\K$, while $\trop(\Pf{k}{n})$ and $\trop^+(\Pf{k}{n})$ are (probably) not. The first statement is over an arbitrary field.
The second statement is stronger, but it makes sense only over fields of characteristic zero.

\begin{proof}
Let $v\in \PVplus{k}{n}$. We want to show that $v\in \trop(\Pf{k}{n})$. In fact, it is enough to show this under the assumption that $v$ is sufficiently generic (within $\PVplus{k}{n}$), since $\trop(\Pf{k}{n})$ is closed. By Theorem~\ref{thm:main}, being generic in $\PVplus{k}{n}$ implies that $v=d(w)$ for a $w$ with support equal to a $k$-triangulation. By Lemma~\ref{lemma:balanced} the latter implies that the initial form of every Pfaffian for the weight vector $v$ vanishes at the point $(1,\dots,1)$. Since Pfaffians are a Gr\"obner basis for $v$ by Theorem~\ref{thm:groebner}, we have that
\[
(1,\dots,1) \in V(\ini_v(I_k(n))).
\]
This clearly implies that $\ini_v(I_k(n))$ contains no monomials (over an arbitrary field) and that it does not contain polynomials with all coefficients real and of the same sign (over fields of characteristic zero).
\end{proof}

Putting together Theorem~\ref{thm:main} and Corollary~\ref{cor:balanced} we conclude Theorem~\ref{thm:V+}.

\begin{remark}
\label{rem:tree_metrics}
Since Pfaffians of degree two coincide with the 3-term Pl\"ucker relations that generate the Grassmannian $\Gr2n$, we have that $\Pf{1}{n}=\Gr2n$ and that $\PV{1}{n}$ equals  the \emph{ Dressian} $\mathcal{D}r_2(n)$ (the tropical prevariety defined by quadratic Pl\"ucker relations~\cite[Section 4.4]{MacStu}). 

It was proven in \cite{SpeStu}
that $\mathcal{D}r_2(n)=\trop(\Gr2n)$ (equivalently, $\PV{1}{n}=$ $\trop(\Pf{1}{n})$,  by showing that $\trop(\Gr2n)$ also coincides with the space $\Tree_n$ of \emph{tree metrics} for trees with $n$ leaves.
The proof is reproduced in \cite[Theorem 4.3.3]{MacStu} and the idea of it is the following: 
The tropical hypersurface corresponding to the Pfaffian of degree two (or the 3-term Pl\"ucker relation) of a certain $U \subset\binom{[n]}{4}$ equals the solution set of:
\[
v_{i,j}+ v_{k,l}
\le \max\{v_{i,k}+ v_{j,l}, \ v_{i,l}+ v_{j,k}\},
\quad \forall \{i,j\}\in \binom{U}{2}.
\]
These relations (taken for all $U$) are exactly the \emph{four-point conditions} that characterize tree metrics~\cite{Buneman}. Hence,  $\trop(\Pf{1}{n}) \subset \PV{1}{n} = \Tree_n$.
For the converse, for any given (generic) $v\in  \Tree_n= \PV{1}{n}$ there is a ternary tree $T$ with nonnegative weights $w$ on its edges and realizing $v$ as a tree metric. By relabelling its leaves, we can assume that $T$ is the dual tree of a certain triangulation of the $n$-gon. Hence, $v$ coincides (after relabelling, but this does not change $\trop(\Pf{1}{n})$) with the $d(w)$ of Definition~\ref{defi:d(w)} for this choice of weights.  Theorem~\ref{thm:main} and Corollary~\ref{cor:balanced}  then imply that $v\in \PVplus{1}{n} \subset \trop(\Pf{1}{n})$. 
\end{remark}

We do not have a concrete example showing that $\PV{2}{n} \ne \trop(\Pf{2}{n})$ for any $n$, nor $\PV{k}{2k+3} \ne \trop(\Pf{k}{2k+3})$ for any $k$, but the above proof cannot work for $k\ge 2$ since not every cone in $\PV{k}{n}$ can be sent to $\PVplus{k}{n}$ by a relabelling of the vertices. This is illustrated in the following example.

\begin{example}
\label{example:positive}
Let $n=6$ and $k=2$. Observe that  $\PV{2}6=\trop(\Pf{2}{6})$ since it is a hypersurface.

Consider the $v\in \R^{\binom{[6]}{2}}$ defined by
\[
v_{1,3}=v_{2,3}=v_{2,4}=v_{4,5}=v_{5,6}=v_{1,6}=1,
\]
and  $v_{i,j}=0$ for every other $i,j$. This $v$ lies in $\PV{2}6$ since it gives maximum weight to (exactly) two matchings, namely $\{13, 24, 56\}$ and $\{23,45,16\}$. 

Since the first matching is negative and the second one is positive, we have that $v\in \trop^+(\Pf{2}{6})$.
Since the two matchings do not differ by a single swap, part (2) of Theorem~\ref{thm:main} implies that no relabelling sends $v$ to $\PVplus{2}6$.
\end{example}

The example also shows that  $\trop^+(\Pf{k}{n})$ is not contained in the Gr\"obner cone of $k+1$-crossings, but that is also easy to achieve with the following simpler example: let $v_{13}=1$ and every other $v_{ij}=0$. For any $k\ge2$ and every $n\ge 6$ this gives a point in $\trop^+(\Pf{k}{n})$ (in every maximum matching of size 3 we can swap the two edges of weight zero to get a maximum matching of the opposite sign) that is not in the Gr\"obner cone (in any $U$ containing $\{1,3\}$ the matching using $\{1,3\}$ has weight larger than the 3-crossing).

\subsection{Proof of Lemma~\ref{lemma:balanced}}

\label{sec:balanced}

In the following result we call an \emph{accordion} any sequence $E_1,\dots,E_m$ 
of edges from $\bnn$ such that: (a) For every $i=1,\dots,n-1$, $E_i$ and $E_{i+1}$ share a vertex; (b) For every $i=2,\dots,n-1$, the points $a= E_{i-1}\setminus E_i$
and $b=E_{i+1}\setminus E_i$ lie on opposite sides of the line containing $E_i$. (Put differently, $\{a,b\}$ crosses $E_i$).

 The only property of $k$-triangulations that we need in what follows (apart from the fact that they are $(k+1)$-free) is:

\begin{lemma}
\label{lemma:accordion}
Let $T$ be a $k$-triangulation of the $n$-gon, for some $k$. Then, every two edges of $T$ that do not cross are part of an accordion contained in $T$.
\end{lemma}

\begin{proof}    
    Let $E=\{e_1,e_2\}$ and $F=\{f_1,f_2\}$ be the two edges of $T$; we assume without loss of generality that $1\le e_1 \le f_1 < f_2 \le e_2\le n$. We will use induction on 
    $\min\{f_1-e_1 , e_2-f_2\}$, taking as base cases those with $e_1=f_1$ or $e_2=f_2$, which are trivial. That is, we suppose then that $E$ and $F$ have no endpoints in common.
    
    If $\{e_1,f_2\}\in T$, we are done. Suppose on the contrary that $\{e_1,f_2\}\notin T$. Then there is a $k$-crossing $K$ in $T$ that crosses that edge. That is, $K\cup\{e_1,f_2\}$ is a $(k+1)$-crossing contained in $T\cup \{e_1,f_2\}$. Let $G$ be the edge next to $\{e_1,f_2\}$ in the positive direction in this $(k+1)$-crossing. If $G$ crossed $E$ (resp. $F$), then every edge in $K$ would cross $E$ (resp. $F$), which would imply that $T$ contains the $(k+1)$-crossing $K\cup\{E\}$ (resp. $K\cup\{F\}$). Thus, $G$ does not cross any of $E$ or $F$. Inductive hypothesis implies that $T$ contains an accordion from $E$ to $G$ and an accordion from $G$ to $F$, and the union of these two accordions is an accordion from $E$ to $F$.
\end{proof}

We now consider a subset $U\in\binom{[n]}{2k+2}$ and $v=d(w)\in \Grob_k(n)$ sufficiently generic. Genericity implies, by Theorem~\ref{thm:main},  that the support of $w$ is a certain $k$-triangulation $T$.
For each edge $E\in T$ we call \emph{length of $E$ with respect to $U$} and denote it $\ell_U(E)$ the smallest size of the two parts of $U$ separated by $E$. If both parts are equal (that is, if $\ell_U(E)=k+1$) we say that $E$ is a diameter of $U$.

For a matching $M$ of $U$ and an edge $E$ of $T$ we denote by $c_M(E)$ the number of edges of $M$ that cross $E$.
Remember that, $v$ being in the Gr\"obner cone, the maximum weight among matchings of $U$ is the weight of the $(k+1)$-crossing.

\begin{lemma}
\label{lemma:length}
Let $M$ be a matching of $U$. Then, $M$ is of maximum weight with respect to $v$ if, and only if, for every $E\in T$ we have that $\ell_U(E) = c_M(E)$.
\end{lemma}

\begin{proof}
    Observe that the equality  $\ell_U(E) = c_M(E)$ holds for the case when $M$ is the $(k+1)$-crossing, and that, for arbitrary $M$, knowing which edges of $T$ cross each edge of $M$ is enough to compute the weight of $M$. This shows the sufficiency of $\ell_U(E) = c_M(E)$.
    
    Now suppose that $\ell_U(E) > c_M(E)$ for some edge $E\in T$.
    Take a vector $w'$ obtained setting $w_E$ to its minimum possible value while staying in $\Grob_k(n)$. For $v'=d(w')$, the $(k+1)$-crossing is still the maximum weight matching, so
    \[\sum_{E\in T}w'_Ec_M(E)\le\sum_{E\in T}w'_El_U(E) \Rightarrow\sum_{E\in T}w'_E(l_U(E)-c_M(E))\ge 0\]
    Our condition in $w$ implies that $w_E>w'_E$, so
    \[\sum_{E\in T}w_E(l_U(E)-c_M(E))>0 \Rightarrow \sum_{E\in T}w_Ec_M(E)<\sum_{E\in T}w_El_U(E)\]
    Hence, $M$ is not of maximum weight.
\end{proof}

For the rest of this section, we collapse the $n$-gon to a $(2k+2)$-gon by leaving only the sides labelled by $U$; that is, by contracting all edges $E$ with $\ell_U(E)=0$. We denote $T_U$ the subgraph of $K_{2k+2}$ obtained from $T$ after this collapse. We introduce the following partial order among edges of $T_U$ (or, in fact, among edges of $K_{2k+2}$): $E$ and $F$ are incomparable if they either cross or are separated by a diameter of $U$, and if they are comparable then they are ordered according to their $\ell_U$. 

Observe that both $\ell_U(E)$ and $c_M(E)$ depend only on the class of $E$ in $T_U$. 
Thus, Lemma~\ref{lemma:length} needs only to be checked in $T_U$ and not in $T$. 
(That is, only one representative edge of $T$ for each class in $T_U$ needs to be checked). But even more is true. Let $T_U^{\max}$ be the set of edges of $T_U$ that are maximal (within $T_U$) for this order. 

\begin{lemma}
\label{lemma:maximal}
Let $M$ be a matching of $U$. If $\ell_U(E) = c_M(E)$ holds for the edges in $T_U^{\max}$ then it holds for all edges in $T_U$, hence in $T$.
\end{lemma}

\begin{proof}
    Let $E<E'$ be two edges of $T_U$ and suppose that $\ell_U(E') = c_M(E')$. Then, the edges of $M$ that cross $E'$ match the $\ell_U(E')$ edges of the $(2k+2)$-gon on the shorter side of $E'$ to the same number of edges on the longer side (if $E'$ is a diameter it does not matter which side we call ``short''). By definition of $E<E'$, the smaller side of $E$ is contained in the smaller side of $E'$, so the same holds for $E$ and  $\ell_U(E) = c_M(E)$.
\end{proof}

This last lemma suggests we should look at properties of $T_U^{\max}$:

\begin{lemma}
\label{lema:unused}
\begin{enumerate}
    \item Every two edges in $T_U^{\max}$ either cross each other or share a vertex.
    
    \item There is a vertex of the $(2k+2)$-gon not used in $T_U^{\max}$.
\end{enumerate}
\end{lemma}

\begin{proof}
For part (1) we use Lemma~\ref{lemma:accordion} and the observation that the passage from $T$ to $T_U$ preserves accordions. In particular, every two edges of $T_U$ that do not cross are part of an accordion in $T_U$. Only two of the edges of an accordion contained in $T_U$ can be in $T_U^{\max}$, and they share a vertex; hence, every two edges in $T_U^{\max}$ that do not cross share a vertex.

This finishes the proof of part (1) and gives us two possibilities: 
\begin{itemize}
    \item If all the edges in $T_U^{\max}$ mutually cross, then $T_U^{\max}$ is a $j$-crossing for some $j<k+1$. Hence, at least one (in fact at least two) of the $2k+2$ vertices of the $(2k+2)$-gon are not used.
    \item If two edges $E$ and $F$ of $T_U^{\max}$ share a vertex $p$, then none of them is a diameter and, in fact, they are on opposite sides of the diameter using $p$.  Then the opposite vertex $q$ of that diameter is not used in $T_U^{\max}$ because it is impossible for an edge with an end-point in $q$ other than the diameter itself to cross or share a vertex with both of $E$ and $F$.
\end{itemize}
In both cases we have a proof of part (2).
\end{proof}

\begin{lemma}
\label{lemma:unused2}
Let $p$ be a vertex of the $(2k+2)$-gon not used in $T_U^{\max}$. Let $a$ and $b$ be the elements of $U$ next to $p$. Then, no maximal matching of $U$ matches $a$ to $b$.
\end{lemma}

\begin{proof}
    To seek a contradiction, suppose that $M$ is a maximal matching and that $\{a,b\}\in M$. Let $\{c,d\}\in M$ be another edge in the matching. By Lemmas~\ref{lemma:length} and \ref{lemma:maximal}, the swaps $\{a,c\}, \{b,d\}$ and $\{a,d\}, \{b,c\}$ cross $T_U^{\max}$ exactly as many times as the original pair of edges $\{a,b\}, \{c,d\}$; that is, as many times as the single edge $\{c,d\}$ (since $\{a,b\}$ does not cross $T_U^{\max}$). This implies that no edge of $T_U^{\max}$ has $a$ and $b$ on one side and $c$ and $d$ on the other side. 
    
    Now, since all edges of $T_U^{\max}$ have $a$ and $b$ on the same side, we conclude that this side must contain one of $c$ or $d$ for every $\{c,d\}\in M$ other than $\{a,b\}$. In particular, for every $E\in T_U^{\max}$ the side of $E$ containing $a$ and $b$ has length at least $k+2$ (it contains $a$, $b$ and one vertex of each of the other $k$ edges in $M$). This gives the following contradiction: Let $p'$ be one of the vertices of the $(2k+2)$-gon next to $p$. The edge $\{p,p'\}$ is in $T_U$, since every boundary edge of the $2k+2$-gon is. Hence, there must be an edge in $T_U$ that is greater than  $\{p,p'\}$ in the partial order, and that edge can have length at most $k+1$ on the side containing $a$ and $b$.
\end{proof}

We are now ready to prove Lemma~\ref{lemma:balanced}:

\begin{proof}[Proof of Lemma~\ref{lemma:balanced}]
Let $p$ be a vertex of the $(2k+2)$-gon not used in $T_U^{\max}$, which exists by Lemma~\ref{lema:unused}. Let $a$ and $b$ be the first elements of $U$ on both directions starting at $p$. 

Let us denote by $\mathcal M$ the set of matchings of $U$ not using the edge $\{a,b\}$. This contains all matchings of maximum weight by Lemma~\ref{lemma:unused2}. Consider the map $\phi: \mathcal M\to \mathcal M$  that takes each matching $M\in \mathcal M$ and swaps in it the edges that contain $a$ and $b$ in the way that does not produce the pair $\{a,b\}$. This map is well-defined since there are three possible matchings among four vertices and we are excluding one of them.
We have that:
\begin{itemize}
    \item The map $\phi$ is obviously an involution.

    \item The map $\phi$ sends matchings of maximum weight to matchings of maximum weight by Lemmas~\ref{lemma:length} and \ref{lemma:maximal}, since every edge of $T_U^{\max}$ leaves $a$ and $b$ on the same side.
    
    \item If $a'$ and $b'$ are the elements of $U$ matched to $a$ and $b$ in a certain matching $M$ then the matching of of $a,b,a',b'$ that has a crossing is involved in the swap from $M$ to $\phi(M)$ (because the matching that is \emph{not} involved in the swap is $\{a,b\}, \{a',b'\}$, which does \emph{not} have a crossing). Hence, $M$ and $\phi(M)$ have opposite parity, by Lemma~\ref{lema:swap}.
\end{itemize}

Putting these facts together we conclude that $\phi$ restricts to a bijection between the odd and the even matchings of $U$ of maximum weight.
\end{proof}

\section{Catalan-many associahedra; recovering the $\g$-vector fan}
\label{sec:Catalan}

In this section we look at the case $k=1$ and show how to project $\PVplus{1}{n}$ isomorphically to the associahedron $\OvAss{1}n$. 
Throughout the section let $T\subset\binom{[n]}2$ be an arbitrary triangulation of the $n$-gon, that we call the \emph{seed triangulation}.
Then:

\begin{lemma}
\label{lemma:projection}
    For every $(v_{i,j})_{i,j} \in \PVplus{1}{n}$, knowing the entries of $v$ corresponding to $T$ we can recover all other entries.
    That is,
    the projection $\pi:\PVplus{1}{n} \to \R^{T}\cong \R^{2n-3}$  that restricts each vector $(v_{i,j})_{i,j}$ to the entries with $\{i,j\} \in T$ is injective. 
\end{lemma}

\begin{proof}
Let $v \in \PVplus{1}{n}$ and let us see that we can recover the entry $v_{i,j}$ for any $\{i,j\} \in \binom{[n]}2$, knowing the entries of $v$ corresponding to edges of $T$.

The proof is by induction on the number of triangles of $T$ crossed by $\{i,j\}$. If only two triangles are crossed, then $\{i,j\}$ is the only unknown entry from the quadruple $U=\{i,j,k,l\}$ consisting of those two triangles, and the edges $\{i,j\}$ and $\{k,l\}$ cross. Since $d \in \PVplus{1}{n}$, we have that the maximum weight among the three matchings in $U$ is attained by $\{ij,kl\}$ and at least one of the other two matchings, so we can write:
\[
v_{i,j} = \max\{v_{i,k}+v_{j,l}, v_{i,l}+v_{j,k}\} - v_{k,l}.
\]

If $\{i,j\}$ crosses more than two triangles, let $\{k,i,l\}$ be the triangle incident to $i$ and crossed by $\{i,j\}$. By inductive hypothesis, all the entries among the $4$-tuple $\{i,j,k,l\}$ are known except for the entry $\{i,j\}$, so we can recover $v_{i,j}$ with the same formula as above.
\end{proof}

That is, $\pi$ embeds $\PVplus{1}{n}$ as a full-dimensional fan $\pi(\PVplus{1}{n})\subset \R^{T}\cong \R^{2n-3}$. We are interested in a second projection
\[
\phi: \R^{T} \to \R^{\overline T} \cong \R^{n-3}
\]
that sends the irrelevant face of $\pi(\PVplus{1}{n})$ to zero, so that $\phi(\pi (\PVplus{1}{n})) $ is a fan isomorphic to the link of the irrelevant face in $\pi(\PVplus{1}{n})$, that is, isomorphic to $\OvAss{1}n$, the normal fan of the associahedron. Here, $\overline T$ denotes the relevant part (the $n-3$ diagonals) of $T$.

\begin{corollary}
    \label{coro:fan}
    The projection 
   \[
   \phi\circ\pi: \PVplus{1}{n} \to \R^{\overline T} \cong \R^{n-3}
   \]
    gives a realization of the associahedron $\OvAss{1}n$ as a complete fan.
\end{corollary}

\begin{proof}
    This projection is conewise linear (linear in each cone). After normalizing, it becomes a continuous map from the $(n - 4)$-dimensional sphere $\OvAss{1}{n}$ to the unit sphere in $\R^{n - 3}$ and, by Lemma~\ref{lemma:projection}, it is injective. Since every injective continuous map from a sphere to itself is a homeomorphism, $\phi(\pi(\PVplus{1}{n}))$ is a  complete fan.
\end{proof}

\begin{remark}
Lemma~\ref{lemma:projection} and its Corollary~\ref{coro:fan} do not hold for $k\ge 2$. In fact, suppose we take $T$ to be any $k$-triangulation containing all the edges of the form $(1,i)$ and $(2,i)$, which exists since $k\ge 2$. Consider now the cone corresponding to a $k$-triangulation $T'$ that does not use a certain edge $(1,i)$. In this cone we have $w_{1,i}=0$ and hence
\[
v_{1,i}+v_{2,i+1} = v_{1,i+1} + v_{2,i}.
\]
Thus, the projection $\pi$ is not injective; it collapses the cone of $T'$ to lower dimension. 
\end{remark}

We now want to give a more explicit description of the fans in \ref{coro:fan}, that is, explicit coordinates for the ray corresponding to each diagonal $\{i,j\}\in \binom{[n]}2$. For this we define the following $\g$-vector of $\{i,j\}$ with respect to the seed triangulation $T$. 

\newcommand{\zig}{ {{\mathsf Z}}}
\newcommand{\zag}{ \reflectbox{{$\mathsf Z$}}}

Remember that $T$ is embedded as a true triangulation using the vertices of our $n$-gon, while the diagonals $\{a,b\}\in \binom{[n]}2$ corresponding to coordinates in our ambient space correspond to pairs of sides. For any given diagonal $\delta$
we define the following \emph{crossing sign} of $\{a,b\}$ with respect to $\delta$ and the $\g$-vector of $\{a,b\}$ with respect to $T$ as follows:

\begin{definition}[See \protect{\cite[Proposition 33]{HPS}} or \protect{\cite[Definition 1.1]{HPS-fpsac}}]
Let $\delta$ be a diagonal in $\overline T$ and $\{a,b\}\in \binom{[n]}2$. Let $q(\delta)$ be the quadrilateral in $T$ consisting of $\delta$ and its two adjacent triangles. We define the \emph{crossing sign of $\{i,j\}$ with respect to $\delta$ in $T$}
\[
\varepsilon(\delta\in T, \{a,b\}):=
\begin{cases}
  +1 & \text{ if $\{a,b\}$ crosses $q(\delta)$ as a $\zig$}\\
  -1 &\text{ if $\{a,b\}$ crosses $q(\delta)$ as a $\zag$}\\
  0 & \text{otherwise}
\end{cases}
\]

We define the \emph{$\g$-vector of $\{a,b\}$ with respect to $T$} as 
\[
\g(T,\{a,b\}) := (\varepsilon(\delta\in T,\{a,b\}))_{\delta\in \overline{T}}
\]
\end{definition}

Observe that the ``otherwise'' in the definition of the crossing sign includes all cases in which $\delta$ does not separate $a$ from $b$, but it also includes the case in which $\{a,b\}$ ``cuts a corner'' of $q(\delta)$. Put differently, $\varepsilon(\delta\in T, \{a,b\})$ is nonzero if and only if two opposite sides of the quadrilateral $q(\delta)$ separate $a$ from $b$, and its sign depends on which two.

For example, for the following triangulation and the edge $\{2,6\}$, we have
\begin{center}
    \begin{align*}
	    \varepsilon(\{1,3\}\in T,\{2,6\})=\varepsilon(\{0,5\}\in T,\{2,6\}) & =1 \\
	    \varepsilon(\{0,3\}\in T,\{2,6\}) & =-1 \\
	    \varepsilon(\{0,4\}\in T,\{2,6\})=\varepsilon(\{5,7\}\in T,\{2,6\}) & =0
	\end{align*}
	
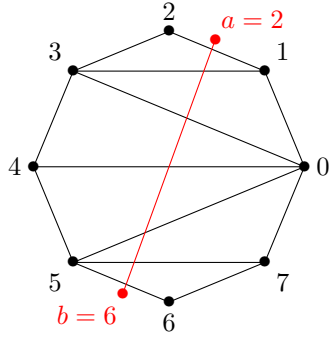
\begin{tikzpicture}[inner sep=0.2ex]
	\node[label=0:0] (p0) at (0:1.8cm) {$\bullet$};
	\node[label=45:1] (p1) at (45:1.8cm) {$\bullet$};
	\node[label=90:2] (p2) at (90:1.8cm) {$\bullet$};
	\node[label=135:3] (p3) at (135:1.8cm) {$\bullet$};
	\node[label=180:4] (p4) at (180:1.8cm) {$\bullet$};
	\node[label=225:5] (p5) at (225:1.8cm) {$\bullet$};
	\node[label=270:6] (p6) at (270:1.8cm) {$\bullet$};
	\node[label=315:7] (p7) at (315:1.8cm) {$\bullet$};
	\draw (p0.center)--(p4.center) (p0.center)--(p3.center)--(p1.center) (p0.center)--(p5.center)--(p7.center);
	\draw (p0.center)--(p1.center)--(p2.center)--(p3.center)--(p4.center)--(p5.center)--(p6.center)--(p7.center)--(p0.center);
	\draw[color=red] (70:1.8cm) node[label=70:{$a=2$}] {$\bullet$} -- (250:1.8cm) node[label=250:{$b=6$}] {$\bullet$};
	\end{tikzpicture}
\end{center}
\begin{remark}
\label{rem:g-vectors} 
\begin{itemize}
    \item $\g(T,\{a,a+1\}) =0$ for every $a$, since $\{a,a+1\}$ cannot cross two opposite sides of any quadrilateral.
    \item If $\{a,b\}$ is in $\overline T$ then $\varepsilon(\delta\in T, \{a,b\}) =1$ if $\delta=\{a,b\}$ and is zero otherwise. 
    Similarly $\varepsilon(\delta\in T, \{a+1,b+1\}) =-1$ if $\delta=\{a,b\}$ and is zero otherwise. 
    
    Thus, in this case the  $\g$-vector $\g(T,\{a,b\})$
    equals the corresponding standard basis vector and $\g(T,\{a+1,b+1\})$ equals its opposite.
    
    \item In general, $\g(T,\{a,b\})$ has the following interpretation: The edges of $T$ crossed by $\{i,j\}$ form an \emph{accordion} in the sense of Section~\ref{sec:balanced}. The signs in the vector $\g(T,\{a,b\})$ record at which edges the accordion turns left or right. In particular, the $\g$-vector is zero for edges of $T$ that are not in the accordion, but also for those in which the accordion `does not turn'.
\end{itemize}
\end{remark}

This definition of $\g$-vectors, which we have taken from
Hohlweg, Pilaud and Stella~\cite{HPS}, is a specialization to the disc of the \emph{shear coordinates} described for arbitrary surfaces in \cite{FomThu}.
They consider the $\g$-vector fan obtained considering as cones all the possible clusters (which, in type $A$ are the triangulations) and taking as generators the $\g$-vectors for a fixed but arbitrary seed triangulation $T$. The main result of~\cite{HPS} is that these fans are polytopal. It turns out that these fans are linearly isomorphic to the ones of Corollary \ref{coro:fan}:

\begin{theorem}
\label{thm:g-vector}
Let $\Sigma_T = im(\phi\circ \pi)$ be the associahedral fan in $\R^{n-3}$ of Corollary \ref{coro:fan} for a certain seed triangulation $T$.

In the basis of $\R^{n-3}$ consisting of the rays corresponding to the diagonals of $\overline T$ we have that $\Sigma_T$ equals the $\g$-vector fan of $T$.
\end{theorem}

\begin{proof}
    For each $(i,j)\in \binom{[n]}2$ let  $W_{i,j}$ be the generator of the orthant $\Sep_n$ corrresponding to a certain $\{i,j\}$. That is, $W_{i,j} = d(w)$ for the vector $w$ with $w_{i,j}=1$ and $w_{i',j'}=0$ if $\{i',j'\}\ne \{i,j\}$. We think of $W_{i,j}$ as the standard basis vector in the coordinates $w_{i,j}$, and let $V_{i,j}$ be the standard basis vector in the coordinates $v_{i,j}$ that we have been using so far. The $W_{i,j}$ are also the generators for the fan structure in $\PVplus{k}{n}$, so that $\phi\circ\pi(W_{i,j})$ is the corresponding generator of $\phi\circ\pi(\PVplus{1}{n})$.

The relations in Definition~\ref{defi:d(w)}, which express the coordinates $v$ in terms of the coordinates $w$, get transposed to the following relations among the vectors $W_{i,j}$ and $V_{a,b}$:
\begin{align}
\label{eq:W}
W_{i,j} = \sum_{\substack{\{a,b\}\in \binom{[n]}2\\
i < b \le j < a \le i}}
V_{a,b}.
\end{align}

Observe that the projections $\pi$ and $\phi$ are defined, respectively, by what they do to the vectors $V$ and $W$, respectively. $\pi$ sends 
 $V_{i,j}$ to zero if $\{i,j\}\not\in T$, and $\phi$ sends $\pi(W_{i,i+1})$ to zero for every $i$.
For simplicity, for each vector $V\in \R^{\binom{[n]}2}$ we will denote by $\overline V :=\phi(\pi(V))\in \R^{n-3}$, and the same for $\overline W$.

    Let $\{i,j\}$ be a diagonal of $T$. We then have
    \[
    \overline W_{i,j} + \overline W_{i+1,j+1} =
    \overline W_{i,i+1} + \overline W_{j,j+1} =
    0,
    \]
    where the first equality comes from Equations~\eqref{eq:W} taking into account that the only edges of $T$ crossing $\{i,j\}$ or $\{i+1,j+1\}$ are those with an end-point in $i$ or $j$, and each of them crosses 
    $\{i,j\}$ and $\{i+1,j+1\}$ the same number of times as it crosses $\{i,i+1\}$ or $\{j,j+1\}$. (Namely, they all cross once except for the edge $\{i,j\}$ which crosses twice). The second equality comes from the fact that $\phi(\pi(W_{i,i+1})=0$ for every $i$. 
    Thus we have
    \[
    \overline 
    W_{i,j} = - \overline W_{i+1,j+1}
    \]
    for each diagonal $\{i,j\}$ of $T$.
    
    Now,
    let $a$ and $b$ be two sides of the $n$-gon and consider the accordion in $T$ between $a$ and $b$. Let $\{i_1,j_1\}, \dots,  \{i_\ell,j_\ell\}$ be the diagonals of $T$ at which the accordion has an ``inflection point'' (it changes from turning left to turning right, or viceversa, that is, $\{a,b\}$ crosses $\{i_m,j_m\}$ as a $\zig$ or a $\zag$ alternatively). 
    The statement we want to prove is that
    \begin{align}
    \label{eq:zigzag}
    \overline W_{a,b} =\sum_{\delta\in \overline{T}}\varepsilon(\delta\in T,\{a,b\})\overline W_\delta= \sum_m \varepsilon(\{i_m,j_m\}\in T,\{a,b\}) \overline W_{i_m,j_m}
    \end{align}
    Note that $-\overline W_{i_m,j_m}$ equals $\overline W_{i_m+1,j_m+1}$, so we are taking the sum of the edges in the zigzag turned in the direction of the path. Indeed, the sum in the right-hand side includes three times the diagonals $\{i_m,j_m\}$, twice the rest of diagonals in the accordion and once the rest of edges with an end-point in vertices where an $\{i_m,j_m\}$ meets the next one.
    Subtracting the irrelevant $\overline{W}$'s for these vertices, we get exactly once the diagonals separating $a$ and $b$, and only them. 
\end{proof}

Although, as said above, polytopality of the $\g$-vector fans is proven in \cite{HPS}, we include here a proof for completeness:

\begin{proposition}[\cite{HPS}]
\label{prop:poly}
For every $T$, the fan $\Sigma_T$ is polytopal.
\end{proposition}

\begin{proof}
    Once we have explicit vectors $\overline W_{i,j}$ (in the basis given by the diagonals of $T$) for the rays in our fan, it remains to find right-hand sides $(b_{ij})_{i,j} \in \R^{\binom{[n]}2}$ 
    for the equations 
    \[
    \overline W_{i,j} X \le b_{ij}, 
    \]
    to get a realization of the polytope with this normal fan. Remember that a right-hand side vector $(b_{ij})_{i,j}$ is valid for a given complete fan if and only if for every pair of adjacent facets in the fan, we have that 
    \begin{equation}\label{ineq}
		\sum_{(i,j)\in C}\omega_{ij}(C)b_{ij}>0,
	\end{equation}
	where $\omega_{ij}$ are the coefficients of the unique circuit $C$ contained in the generators of those two cones, with the sign of $\omega$ chosen positive in the rays that are not common to the two cones. See, e.g., \cite[Section 5]{CSZ} for details.
	
	We are going to show that the choice of right-hand sides $b_{i,j} = (j-i)(n+i-j)$ (where we assume $i<j$) is valid.

     In the associahedral fan all circuits contained in adjacent pairs are supported on the $K_4$ graph consisting of the diagonals that are flipped and the convex quadrilateral containing them. 
     
     Given such a $K_4$ with vertices $\{a,b,c,d\}$
     with $a<b<c<d$, either there is no edge in $T$ with an end between $a$ and $b$ and another between $c$ and $d$, or there is no edge with an end between $b$ and $c$ and another between $a$ and $d$. 
     
     Suppose the first case holds. Then all the diagonals in $T$ cross the pairs $ac$ and $bd$ the same number of times than $ad$ and $bc$. $\overline W_{ac}$ is the sum of the $\overline D$ for the diagonals in $T$ that cross $ac$, and the same for the other pairs. This implies
	\[
	\overline{W}_{ac}+\overline{W}_{bd}=\overline{W}_{ad}+\overline{W}_{bc}
	\]
	Thus,  (\ref{ineq}) becomes
	\[(c-a)(n-c+a)+(d-b)(n-d+b)-(d-a)(n-d+a)-(c-b)(n-c+b)>0\]
	which reduces to $2(d-c)(b-a)>0$.

	In the second case, a similar argument rewrites (\ref{ineq}) as
	\[(c-a)(n-c+a)+(d-b)(n-d+b)-(b-a)(n-b+a)-(d-c)(n-d+c)>0\]
	which reduces to $2(c-b)(n+a-d)>0$.
\end{proof}

\begin{remark}
The construction in this proof is the same one of  \cite{HPS} taking as ``exchange submodular'' function the function $(j-i)(n-j+i)$. In particular, the vertices obtained for the associahedra are the 
$\mathbf{c}$-vectors, also introduced implicitly in \cite{FomZel}.
\end{remark}

\begin{remark}
From the perspective of cluster algebras, associahedra are the type $A$ case of the \emph{generalized associahedra} that Fomin and Zelevinsky defined as simplicial spheres and 
 F. Chapoton, S. Fomin and A. Zelevinsky \cite{CFZ} constructed as polytopes, using the so-called $\mathbf{d}$-vector fans for certain seed clusters. %
In type $A$, this construction was generalized by Santos~\cite[Section 5]{CSZ} to obtain Catalan-many associahedra by showing that any triangulation works as seed triangulation in the $\mathbf{d}$-vector construction.
 
 The construction of generalized associahedra via $\g$-vectors instead of $\mathbf{d}$-vectors was first achieved in various special cases by, among others,  Hohlweg-Lange-Thomas~\cite{HLT}, Pilaud-Stump~\cite{PilStu}
 and Stella~\cite{Stella}, before the general case was settled by Hohlweg, Pilaud and Stella in \cite{HPS}.
 
 The associahedral fans obtained obtained by Santos via $\mathbf{d}$-vector fans and by Hohlweg-Pilaud-Stella via $\g$-vector fans from a seed triangulation $T$ have certain similarities:
 
\begin{enumerate}
    \item For each of the $n-3$ diagonals $\{i,j\}\in \overline T$, the ray corresponding to  $\{i,j\}$ is opposite to another ray. That is, the corresponding facets in the associahedron are parallel.

    \item Every other ray can be expressed as a $\{+1,0,-1\}$ combination in the basis given by those $n-3$ rays.
\end{enumerate}

However, they are not the same. In the $\g$-vector fan the ray opposite to a diagonal $\{i,j\}$ of $T$ is $\{i+1,j+1\}$ while in the $\mathbf d$ construction it is the diagonal inserted by the flip of $\{i,j\}$ in $T$. \end{remark}

Summing up, both constructions provide Catalan-many $(n-3)$-asso\-cia\-hedra with $n-3$ pairs of opposite facets, but these pairs are different.
Moreover, this construction has the exponential family of realizations obtained by Hohlweg and Lange \cite{HL} as a subset: they are the realizations obtained when the seed triangulation $T$ does not have interior triangles, that is, when its dual graph is a path, or the initial seed for the \g-vector fan is acyclic.

One could think that there is a variant of \g-vectors for $k>1$. For example, for $k=2$ it is known that multitriangulations are complexes of 5-sided stars \cite{PilSan}, and a \g-vector can be defined assigning different values for $\varepsilon(\{i,j\}\in T,\{a,b\})$ depending on the position of $\{a,b\}$ with respect to the two stars incident to $\{i,j\}$. A priori, the problem would be how to define these $\varepsilon(\{i,j\}\in T,\{a,b\})$ so that they work.
If the two edges cross, there are 4 possible positions for $a$ and the same number for $b$, giving 16 different positions, and the idea would be to use different coefficients as $\varepsilon$ depending on which of the 16 possibilities (or 10, if we mod out symmetry) we are in.

However, this idea only works for concrete cases, and it can not work for $n$ big enough.
\begin{theorem}
For a $k$-triangulation $T$, with $k>1$, if there is an edge not contained in any pair of adjacent stars of $T$, it is impossible to realize the $k$-associahedron as a \g-vector fan, independently of the values chosen for $\varepsilon$.
\end{theorem}
\begin{proof}
    Let $\{a,b\}$ be the edge. We'll show that $\g(a,b)+\g(a+1,b+1)=\g(a,b+1)+\g(a+1,b)$. Then, we can choose a $k$-triangulation that contains these edges (for $k>1$ it will exist), and its cone will not have full dimension.
    
    This equality can be checked component by component. For an edge $\{i,j\} \in T$, either $a$ is not in the two stars delimited by $\{i,j\}$ or $b$ is not. In the first case, $\varepsilon(\{i,j\}\in T,\{a,c\})=\varepsilon(\{i,j\}\in T,\{a+1,c\})$ for any $c$, concretely for $c=b$ and $c=b+1$, and the equality holds for this component. The same happens if $b$ is not in the two stars.
\end{proof}
\begin{corollary}
It is impossible to realize the $k$-associahedron as a \g-vector fan, independently of the values chosen for $\varepsilon$, for $k>1$ and $n$ big enough.
\end{corollary}
\begin{proof}
    Suppose it is possible. Then all edges must be contained in a pair of adjacent stars. There are as many pairs of adjacent stars as relevant edges in $T$, that is, $k(n-2k-1)$. Each pair contains at most $4k$ vertices that form $\binom{4k}{2}$ edges, so we get
    \[\binom{n}{2}\le k(n-2k-1)\binom{4k}{2}\]
    which is false for $n$ big enough.
\end{proof}

\end{document}